\documentclass[12pt]{amsart}
\usepackage{amsmath,amsthm,amsfonts,amssymb,eucal, hyperref,braket,enumitem}
\usepackage{dsfont}
\usepackage{color}

\newcommand{\<}{\langle}
\renewcommand{\>}{\rangle}

\renewcommand{\k}{\kappa}

\renewcommand{\l}{\lambda}

\newcommand{\T}{\Theta}

\newcommand{\oq}{\ {\raise 7pt\hbox{${\scriptstyle\circ}$}}
\kern -7pt{
\hbox{$Q$}}}

\newcommand{\R}{ \mathbb R}

\newcommand{\N}{ \mathbb N}
\newcommand{\I}{ \mathbb I}

\renewcommand{\T}{\mathbb T}

\newcommand {\bb}{\mathbf b}

\newcommand {\bk}{\mathbf k}

\newcommand {\bm}{\mathbf m}
\newcommand {\bl}{\mathbf l}
\newcommand {\bze}{\mathbf 0}

\newcommand {\bn}{\mathbf n}

\newcommand{\cP}{\mathcal P}
\newcommand{\cQ}{\mathcal Q}

\newcommand{\corr}{\mathrm{corr}}

\newcommand{\1}
{{\,\vrule depth3pt height9pt}{\vrule depth3pt height9pt}
{\vrule depth3pt height9pt}{\vrule depth3pt height9pt}\,}

\DeclareMathOperator {\dist} {{dist}}

\newtheorem{thm}{Theorem}[section]
\newtheorem{cor}[thm]{Corollary}
\newtheorem{cla}[thm]{Claim}
\newtheorem{lem}[thm]{Lemma}
\newtheorem{prop}[thm]{Proposition}

\theoremstyle{definition}
\newtheorem{defn}[thm]{Definition}
\newtheorem{example}[thm]{Example}

\newtheorem{rem}[thm]{Remark}

\numberwithin{equation}{section}

\newcommand{\bee}{\begin{equation}}
\newcommand{\ene}{\end{equation}}
\newcommand{\bees}{\begin{equation*}}
\newcommand{\enes}{\end{equation*}}
\newcommand{\bes}{\begin{split}}
\newcommand{\ens}{\end{split}}

\newcommand{\bet}{\begin{thm}}
\newcommand{\ent}{\end{thm}}
\newcommand{\bel}{\begin{lem}}
\newcommand{\enl}{\end{lem}}
\newcommand{\bec}{\begin{cor}}
\newcommand{\enc}{\end{cor}}
\newcommand{\becl}{\begin{cla}}
\newcommand{\encl}{\end{cla}}
\newcommand{\bep}{\begin{proof}}
\newcommand{\enp}{\end{proof}}
\newcommand{\ber}{\begin{rem}}
\newcommand{\enr}{\end{rem}}
\newcommand{\ep}{\varepsilon}

\newcommand{\Z}{\mathbb Z}

\renewcommand{\l}{\left}
\renewcommand{\r}{\right}

\makeatletter
\def\square{\RIfM@\bgroup\else$\bgroup\aftergroup$\fi
  \vcenter{\hrule\hbox{\vrule\@height.6em\kern.6em\vrule}\hrule}\egroup}
\makeatother

 \usepackage{hyperref}

\def\dio{\mathrm{dio}}

\def\cont{\mathtt{Cont}}

\newcommand{\one}{\mathbf{1}}

\usepackage[margin=1.1in]{geometry}

\begin{document}

\title[Pila]
{Smooth perturbation theory and generic Cantor spectrum for multi-dimensional monotone quasiperiodic operators}
\author[I. Kachkovskiy]
{Ilya Kachkovskiy}
\address{Department of Mathematics\\ Michigan State University\\
Wells Hall, 619 Red Cedar Rd\\ East Lansing, MI\\ 48824\\ USA}
\email{ikachkov@msu.edu}
\author[L. Parnovski]
{Leonid Parnovski}
\address{Department of Mathematics\\ University College London\\
Gower Street\\ London\\ WC1E 6BT\\ UK}
\email{l.parnovski@ucl.ac.uk}
\author[R. Shterenberg]
{Roman Shterenberg}
\address{Department of Mathematics\\ University of Alabama, Birminghan\\
University Hall\\10th Ave. S.\\ Birmingham, AL\\ 35294\\USA }
\email{shterenb@uab.edu}

\date{\today}

\begin{abstract}
We consider multi-dimensional quasiperiodic Schr\"odinger operators with monotone sawtooth-type potentials in the perturbative localization regime. We prove that a given gap is closed iff a corresponding (via a natural gap labeling) eigenfunction vanishes at the origin. Using this, we show that for a generic choice of the sampling function $f$, the spectrum of such an operator is a Cantor set, with all possible gaps being open. We also provide a class of functions $f$ for which some gaps are closed. Finally, we observe that certain gaps are more difficult to close than the others.

The proof uses an approach to Rayleigh--Schr\"odinger perturbation theory that provides quantitative estimates of the remainder and can deal with dense point spectra.
\end{abstract}
\maketitle
\section{Introduction}

\subsection{The setting and motivation} Let $f\colon \R\to \R$ be a $1$-periodic function. In this paper, we consider {\it quasiperiodic Schr\"odinger operators} $H(x)=H_f(x)$ acting on $\ell^2(\Z^d)$
\bee
\label{eq_h_def_intro}
(H(x)\psi)(\bn)=\ep(\Delta\psi)(\bn)+V(\bn)\psi(\bn),\quad V(\bn)=f(x+\bn\cdot\omega),\quad \bn\in \Z^d,
\ene
where $\omega\in \R^d$ is a frequency vector with $\{1,\omega_1,\ldots,\omega_d\}$ rationally independent and $x\in \T$ is the quasiperiodic phase. One of the main questions of spectral theory of such operators is what does the spectrum of $H(x)$ look like as a set; in particular, whether it is a Cantor set (i.e. whether it has an empty interior). Depending on the particular class of functions $f$ and frequency $\omega$, there are many different versions on this question, starting from the {\it Ten Martini Problem}. Most of these results are obtained for the case $d=1$ (see, for example, \cite{ABD,DL,ten_martini,ten_martini_dry,beckus,KPS2,BF,Ge2,Li} and references therein). The intuition from the multi-frequency case (\cite{ChuSi,GSW}) suggests that, in the case of analytic $f$, Cantor spectrum for $d\ge 2$ is less likely than for $d=1$. Indeed,
known examples of multi-dimensional operators with Cantor spectrum mostly rely on separable potentials (\cite{HLQZ,DFG}) and are obtained by finding a one-dimensional example with exceptionally ``thin'' spectrum, so that the algebraic sum of the corresponding Cantor sets is still a Cantor set. However, such spectrum is usually considered to be very exotic.

\vskip 1mm

Our main goal is to find a big class of multi-dimensional operators of the form \eqref{eq_h_def_intro} that have Cantor spectra for reasons that cannot be obviously reduced to one-dimensional phenomena. Results of \cite{KPS2} and \cite{KPS3} suggest that bounded {\it Lipschitz monotone} functions
\bee
\label{eq_monotonicity_intro}
f(y)-f(x)\ge y-x,\quad\text{for}\quad 0\le x<y<1
\ene
provide natural candidates for such a class. We will consider {\it sawtooth-type} potentials, i.e. Lipschitz monotone potentials that also satisfy
\bee
\label{eq_sawtooth_intro}
-\infty<f(0)=f(0+0)<f(1-0)<+\infty,\quad f\text{ is continuous on } [0,1),
\ene
for example $f(x)=\{x\}$, in the regime of {\it Diophantine frequency vectors}
\bee
\label{eq_diophantine_intro}
\dist(\bn\cdot\omega,\Z)\ge C_\dio|\bn|^{-\tau_\dio}.
\ene
In \cite{KPS2}, it was shown that, for $d=1$, Cantor spectrum in the setting \eqref{eq_h_def_intro} -- \eqref{eq_diophantine_intro} can be deduced from Anderson localization satisfying a stability property under a certain class of rank one perturbations. Non-perturbative results \cite{JK2} enjoy such stability, and this allowed us to prove that the spectrum is Cantor for $d=1$ and all $\ep>0$. Besides localization, the proof relies on the fact that a cyclic subspace of a one-dimensional Schr\"odinger operator associated to $e_\bze$ cannot be too large.

We can now state one of the main results of the present paper, although we postpone the details on some terminology to later parts of it.
\begin{thm}
\label{th_main_loose}
Fix $\omega$ satisfying \eqref{eq_diophantine_intro} and $0<\ep<\ep_0(C_\dio,\tau_\dio,d)$. Then, for a generic subset of functions $f$ within the class $\eqref{eq_monotonicity_intro}$, \eqref{eq_sawtooth_intro}, the spectrum of the operator \eqref{eq_h_def_intro}  is a Cantor set. Moreover, all possible spectral gaps are open.
\end{thm}
If we want to extend the proof in \cite{KPS2} of the one-dimensional case to the multi-dimensional setting, we are faced with the following questions:
\begin{itemize}
    \item Is there a multi-dimensional localization result that is stable under a class of rank one perturbations that preserve monotonicity condition \eqref{eq_monotonicity_intro}?
	\item The structure of the cyclic subspace associated to $e_\bze$ plays the key role in establishing Cantor spectrum and is closely related to the vanishing of the eigenfunctions of the operator $H(0)$ at the origin. Can one make a more precise statement relating individual eigenfunctions with individual spectral gaps?
	\item Having a very small cyclic subspace associated to $e_\bze$ should be a rare phenomenon (if it ever happens). Can one somehow quantify that rareness in terms of the function $f$?
\end{itemize}
The first question was answered in \cite{KPS3} (see also \cite{Shi}) in a perturbative regime $0<\ep<\ep_0(C_\dio,\tau_\dio,d)$, which is the reason why we restrict ourselves to the perturbative setting here --- see  Proposition \ref{prop_main_kps3} below. Our goal is to address the remaining two questions. The study of the eigenvalues appearing in spectral gaps after applying a rank one perturbation, started in \cite{KPS3}, suggests that the presence of a particular spectral gap is related to non-vanishing of a related eigenvector at the origin. Thus we arrive at the following question: can one control the values of eigenvectors at the origin by making small perturbations of $f$? 

For the sake of obtaining some heuristic intuition, let us momentarily assume that the corresponding eigenvalue is isolated. Then, using some abstract spectral arguments, it is possible to show (see Remark \ref{remark_heuristic}) that one can change the value of the associated eigenvector at the origin from zero to non-zero by a finite sequence of carefully chosen small perturbations of $f$, each of which modifies its value at one point. There are two problems with this approach though: firstly, the eigenvalues of the quasiperiodic operators with sawtooth-type potentials are never isolated and, thus, the traditional perturbation theory is not applicable. Secondly, modifying the value of $f$ at finitely many points would change $f$ into a new function, say $f_1$, that would no longer be monotone and, therefore, the spectrum of $H_{f_1}$ would no longer have to be pure point, which would make applying the perturbation theory to it even more questionable.

The present paper is the result of our attempts to make the above ideas rigorous. Our previous results \cite{KPS} show that in some cases one can  apply analytic Rayleigh -- Schr\"odinger perturbation theory to dense point spectra and obtain non-trivial results on Anderson localization. However, the methods of \cite{KPS} rely heavily on the Maryland-type structure, which is known not to have any spectral gaps \cite{KPS3}. On the other hand, the heuristic gap opening arguments do not require complete analyticity, since they only involve calculations of finitely many derivatives. As a result, we arrive at the question of controlling the remainder in a partial sum of the Rayleigh -- Schr\"odinger perturbation series. Obtaining such control leads to progress in some of the above questions, as well as provides more general smoothness results that may be of independent interest. The following can be considered the main results of the present paper, explained in more detail in the next two subsections.
\begin{itemize}
    \item An abstract version of Rayleigh -- Schr\"odinger perturbation theory (Theorem \ref{th_canonical_expansion}), which requires minimal assumptions on the unperturbed operator and provides diagram representations for the perturbation series and the remainder.
    \item A gap labeling result (Theorem \ref{th_gap_labeling}) for sawtooth-type monotone quasiperiodic operators in the perturbative regime, relating, in particular, the opening of  individual spectral gaps with zeros of corresponding eigenvectors.
    \item The existence of three types of spectral gaps in the multidimensional setting (see Subsection \ref{subsection_principal_gaps}): the gaps with labels associated to coordinate axes which eventually open as $\ep\to 0$ (principal gaps), the gaps whose labels have only two non-zero entries $\pm1$, which also eventually open as a consequence of monotonicity of $f$ (semi-principal gaps), and the remaining non-principal gaps, which may be closed for arbitrarily small values of $\ep$.
    \item Generic Cantor spectrum (Theorem \ref{th_main_cantor}, see also Remark \ref{remark_generic} for the discussion of the associated topologies): under the assumptions \eqref{eq_h_def_intro} -- \eqref{eq_diophantine_intro} and $0<\ep<\ep_0(C_\dio,\tau_\dio)$, for a dense $G_{\delta}$ subset of functions $f$ within the class \ref{eq_monotonicity_intro}, all possible spectral gaps of the corresponding operator are open, in the sense of the labeling established above.
\end{itemize}
\subsection{Results on monotone quasiperiodic operators} Throughout this paper, we assume \eqref{eq_h_def_intro} -- \eqref{eq_diophantine_intro}: that is, Diophantine $\omega$, sawtooth-type function $f$, and $0<\ep<\ep_0(C_\dio, \tau_\dio,d)$, which is the perturbative regime considered in \cite{KPS3} (see also \cite{Shi}). A summary of the associated localization results is provided in Proposition \ref{prop_main_kps3} which, informally, states that there are eigenvalue-eigenvector branches 
\bee
\label{eq_eigenvectors_intro}
H(x)\psi_\bn(x)=E_\bn(x)\psi_\bn(x),
\ene
labeled by lattice points $\bn\in \Z^d$, where $E_\bn(x)$ and $\psi_\bn(x)=\psi_\bn(x,\cdot)$ are uniformly small perturbations, respectively, of the diagonal entries $f(x+\bn\cdot\omega)$ and the standard basis vectors $e_\bn$. In particular, one has a very strong kind of uniform localization.
\vskip 1mm
{\noindent\it $\circ$ Gap labeling and vanishing of the eigenfunctions.} The general theory of gap labeling (the reader is referred to a comprehensive review \cite{DF_label}) predicts that, in the sawtooth-type case, the value of the integrated density of states inside of a spectral gap must be equal to $\{\bn\cdot\omega\}$, where $\bn\in \Z^d$ is called the label of the gap. This prediction is confirmed as part of our Theorem \ref{th_gap_labeling}, although this is done mostly for the lack of references that explicitly cover the class of operators we consider here. Our first main result is Claim 2 in Theorem \ref{th_gap_labeling}, which states that {\it the gap with label $\bn$ is open if and only if the eigenfunction with label $\bn$ and zero phase does not vanish at the origin, i.e. if and only if $\psi_\bn(0,\bze)\ne 0$}. The proof completes the ideas of analyzing eigenvalues in spectral gaps under a class of rank one perturbations, earlier developed in \cite{KPS2}. As a result of this analysis, our Theorem \ref{th_measure_of_spectrum} states that the measure of the spectrum of an operator with a sawtooth-type potential, under the same assumptions, is equal to $f(1-0)-f(0)$, the same as for the operator with $\ep=0$.
\vskip 1mm

{\noindent\it $\circ$ Smoothness in $\ep$ and opening of the principal gaps.} In the earlier work \cite{KPS}, a result similar to Proposition \ref{prop_main_kps3} is obtained by showing the convergence of the Rayleigh--Schr\"odinger perturbation series in a neighborhood of the origin in $\ep$. However, it is only proved for a class of Maryland-type potentials, the kind that never has any spectral gaps. The corresponding eigenvalues and eigenvectors \eqref{eq_eigenvectors_intro} in that case are analytic in $\ep$ in the region of convergence (see also \cite{Bellissard}). While the conclusion of Proposition \ref{prop_main_kps3} somewhat resembles perturbation theory, the actual procedure does not allow to recover any regularity properties in the parameter $\ep$, besides uniform bounds. Our abstract Theorem \ref{th_smoothness} (see also Corollary \ref{cor_uniform_smoothness}) allows to extract infinite smoothness in $\ep$ as a consequence of uniform localization bounds that are contained in Proposition \ref{prop_main_kps3}. Moreover, the derivatives at $\ep=0$ of the corresponding objects, not surprisingly, are equal to the coefficients of the formal Rayleigh--Schr\"odinger series, which leads to Corollary \ref{cor_principal_gaps} stating that every gap with a label of the form $(0,\ldots,0,n,0,\ldots,0)=n\bb_j$ (we will be calling them {\it principal labels}) will be forced to open for small $\ep$ (depending on $n$). The same opening property, but with a stronger reliance on monotonicity, extends to gaps with {\it semi-principal labels} $\pm\bb_i\pm \bb_j$, $i\neq j$. Examples \ref{example_label12} and \ref{example_closing111} demonstrate that gaps with non-principal labels may close for arbitrarily small $\ep$. Although quite simple, the observation that gaps corresponding to principal and non-principal labels behave differently and that it is much easier to close a non-principal gap, seems to be never made before. Since all labels are principal if $d=1$, this distinction is a purely multi-dimensional phenomenon. We believe this distinction should persist for more general classes of $f$.
\vskip 1mm
{\noindent\it $\circ$  Smoothness in $f$.} Another consequence of the general theory is that, in the setting of Proposition \ref{prop_main_kps3}, the eigenvalues and eigenvectors \eqref{eq_eigenvectors_intro} are smooth in $f$, considered as a parameter, in the Gateaux sense. More precisely, for any $g\in C^1$, the function $f+tg$ will be Lipschitz monotone for small $t$, and therefore Proposition \ref{prop_main_kps3} would be applicable for $H_{f+tg}$. Our Corollary \ref{cor_smoothness_g} states that the corresponding eigenvalues and eigenvector branches will be infinitely smooth in $t$ at the origin.

\vskip 1mm
{\noindent\it $\circ$ Cantor spectrum for generic $f$.} The previous smoothness results allows us to control, in a quantitative way, what will happen with eigenfunctions under small $C^1$ perturbations of $f$. Theorem \ref{th_main_cantor}, which is one of the main results of the paper, shows that for every sawtooth-type $f$ one can find $g\in C^1$ with arbitrarily small norm, such that the operator with $f+tg$ will have Cantor spectrum in the strongest possible sense (that is, for every $\bn$ the gap with label $\bn$ will be open). In particular, it shows that {\it the set of functions $f$ with Cantor spectrum is a dense $G_{\delta}$ subset of the set of all sawtooth-type functions with $C^k$ topology}, see Remark \ref{remark_generic}. It is also not hard to modify the argument in order to consider small infinitely smooth perturbations.

\subsection{Smooth perturbation theory}A typical result of analytic perturbation theory \cite{Kato} is that an isolated eigenvalue of a one-parametric analytic operator family is (locally) analytic with respect to the same parameter. In the case of linear dependence on the parameter, such as $\ep$ in \eqref{eq_h_def_intro}, the corresponding Taylor series at the origin is known as the {\it Rayleigh -- Schr\"odinger series} (RSS). In physics literature, the coefficients of these series are typically represented by diagrams \cite{Bloch} and are usually model-specific.

\vskip 1mm

Applications of RSS beyond the case of analytic isolated eigenvalues (perhaps, beyond first order perturbation theory such as in Hellmann--Feynmann theorem) are very scarce. In the previous work \cite{KPS,KKPS}, it is shown that RSS converge for a class of Maryland-type quasiperiodic operators, thus providing an example of analytic perturbation theory successfully applied to dense point spectrum and essentially reaching the limit of our knowledge of the literature.

\vskip 1mm

The proof of Theorem \ref{th_main_cantor} is based on a formal perturbative calculation, stating that perturbing the potential at one point can force an eigenvector to change its value from zero to non-zero, thus potentially opening a gap. The main issue when applying this argument is that such a perturbation is quite likely to destroy pure point spectrum, thus making the whole idea moot. A more careful approach would be a perturbation of $f$ in the form of bump function. However, the main challenge is to make these heuristic arguments quantitative.

\vskip 1mm
The result of the above motivation is a new derivation of Rayleigh--Schr\"odinger perturbation theory with minimal assumptions, see (A1)--(A3) in Subsection \ref{subsection_canonical_expansions}. One can see that it only relies on the unperturbed eigenvalue being simple, the perturbed eigenvector being not completely orthogonal to the unperturbed one, and the actual terms of the perturbation series being well-defined. The main result of this section is Theorem \ref{th_canonical_expansion} which provides expansions for the perturbed eigenvalues and eigenvectors up to any order. These expansions correctly reproduce any number of terms of the original Rayleigh -- Schr\"odinger series and are expressed in the language of diagrams. The advantage of these expansions is that they are always true under the above minimal assumptions and do rely neither on the smallness of the perturbation, nor on any relation between the perturbed and unperturbed eigenvalues, besides the fact that the eigenvectors should not be orthogonal.

\vskip 1mm
Of course, the only reason why such a general result would be true can be contained in the remainder term. Note that the term has correct power of $\ep$ in front of it. The remainder is also expressed in terms of the diagrams. However, the contribution of these diagrams involve both perturbed and unperturbed eigenvector. The choice of wrong eigenvalue/eigenvector branch can (and will) lead to the remainder blowing up, see Remark \ref{remark_arbitrary_eigenfunction}. Still, there is a class of cases considered in Theorem \ref{th_smoothness} and Corollary \ref{cor_uniform_smoothness} where this remainder can be meaningfully estimated. These results can be summarized as follows: {\it existence of an eigenvalue/eigenvector branch with a certain kind of uniform localization without any additional regularity assumptions automatically implies infinite smoothness of this branch}. While the result is stated in an abstract form, for us its main application is in the setting of Proposition \ref{prop_main_kps3}, see Corollaries \ref{cor_principal_gaps} and \ref{cor_smoothness_g}.
\subsection{Structure of the paper}

In Section 2, we discuss the above-mentioned abstract versions of Rayleigh -- Schr\"odinger perturbation theory. The main idea is based on iterating the Lippmann--Schwinger equation \eqref{eq_psi_recursion} and collecting all terms that involve the original eigenvector $\psi_0$ in the terms of highest order. In Subsection 2.3 we introduce the diagram language in the abstract setting and obtain Theorem \ref{th_canonical_expansion} as the main result of that section and the main abstract result of the paper. In Subsection 2.4 we consider the case of an unperturbed operator being diagonal, similar to \cite{KPS}, and further develop the expansion in terms of matrix elements of the operator $A$, again with explicit diagram representations for the remainder.

In Section 3, we state several general consequences, the main ones being Theorem \ref{th_smoothness} and Corollary \ref{cor_uniform_smoothness}, which demonstrate how uniform localization bounds in a relatively abstract setting can be converted into infinite smoothness of eigenvalue/eigenvector branches. In Subsection 3.2 we include some first order calculations in more detail.

In Section 4, we introduce the perturbative setting of monotone quasiperiodic operators following \cite{KPS3} in Proposition \ref{prop_main_kps3} and state the smoothness results contained in Corollaries \ref{cor_principal_gaps}, \ref{cor_smoothness_g}, as well as Theorem \ref{th_gap_labeling}, which is one of the main results of the paper demonstrating the relation between spectral gaps and zeros of the eigenfunctions. In subsection 4.4, we discuss the difference between the gaps corresponding to principal/non-principal labels.

In Section 5, we state and prove the main concrete result of the paper, Theorem \ref{th_main_cantor}, which shows that one can perturb any sawtooth-type potential in order to open all the gaps predicted by Theorem \ref{th_gap_labeling}. The perturbation is constructed as a sequence of careful steps involving bump functions, and the differentiability results from Section 3 play crucial role here.

\subsection{Acknowledgements}I. K. was supported by the NSF grants DMS--1846114, DMS--2052519, and the 2022 Sloan Research Fellowship. L. P. was partially supported by the EPSRC grant EP/V051636/1 and the Leverhulme Trust grant RPG-2023-325. R. S. was supported by the NSF grant DMS-2306327.


\section{Smooth Rayleigh -- Schr\"odinger perturbation theory} 

\subsection{The setting}Let $\mathcal H$ be a Hilbert space, $D$ be a linear self-adjoint (not necessarily bounded) operator on $\mathcal H$, and $A$ be a bounded self-adjoint operator. Assume that $E_0$ is a simple eigenvalue of $D$ with eigenvector $\psi_0$. Suppose that $\psi$ is an eigenvector of $D+\ep A$ that is not orthogonal to $\psi_0$, normalized as follows:
\bee
\label{eq_psi_def}
(D+\varepsilon A)\psi=E\psi,\quad \<\psi,\psi_0\>=1.
\ene
A typical situation in perturbation theory is when $E_0$ is isolated from the rest of the spectrum. In this case, \eqref{eq_psi_def} defines a unique eigenvalue/eigenvector branch, analytic in $\ep$ in some neighborhood of the origin. The corresponding power series is known as the {\it Rayleigh -- Schr\"odinger perturbation series} (RSS). The coefficients of RSS are independent of $\ep$, and each coefficient can be represented as a finite sum over certain kinds of diagrams.

In the case when $E_0$ is not isolated, but the expressions involved in the definitions of the diagrams make sense, RSS can be considered as a formal power series in $\ep$, and the question of its convergence can be delicate \cite{KPS}.

The goal of the present section is to represent $E$ and $\psi$, considered above, as partial sums of the corresponding RSS with the addition of a remainder term. Both partial sums and the remainder will be represented in the form of diagrams. While these representations would not, in general, imply convergence of the infinite series, they can still be used in order to establish smoothness of $E$, $\psi$, and associated quantities.

\subsection{Canonical expansions}\label{subsection_canonical_expansions}
As above, assume that $D$ is a self-adjoint operator on $\mathcal H$, satisfying the following conditions:
\begin{itemize}
    \item [(A1)]$E_0$ is a simple eigenvalue of $D$ with an eigenvector $\psi_0$:
$$D\psi_0=E_0\psi_0,\quad \ker(D-E_0)=\mathrm{span}\{\psi_0\}
$$
\item [(A2)] $\psi$ is an eigenvector of $D+\ep A$ that is not orthogonal to $\psi_0$, normalized as in \eqref{eq_psi_def}:
$$
(D+\varepsilon A)\psi=E\psi,\quad \<\psi,\psi_0\>=1.
$$
\end{itemize}
Denote by $P$ and $R$, respectively, the projection onto the orthogonal complement of $\psi_0$ and the associated projected resolvent:
\bee
\label{eq_projected_resolvent}
P:=\one-\<\cdot,\psi_0\>\psi_0,\quad R:=(D-E_0\one)^{-1}P,\quad R\psi_0:=0.
\ene
\begin{itemize}
    \item[(A3)] $\psi_0$ belongs to the domain of definition of every (non-commutative) monomial constructed from $A$ and $R$.
\end{itemize}
We note that (A1) and (A2) describe the general setting of our question, and $R$ is bounded if and only if $E_0$ is isolated. In our main case of interest that involves non-isolated eigenvalues, we use (A3) in order to be able to claim that all terms of RSS are well-defined. Note that this assumption does not involve the vector $\psi$. In order to be able to use the diagram expression for the remainder, we will sometimes impose an additional condition.
\begin{itemize}
    \item[(A4)]$\psi$ belongs to the domain of definition of every (non-commutative) monomial constructed from $A$ and $R$.
\end{itemize}
Under (A1) -- (A3), one can rewrite the eigenvalue equation \eqref{eq_psi_def} as
\bee
\label{eq_rewrite}
(D-E_0)\psi=(E-E_0)\psi-\ep A\psi.
\ene
Let us further put
\bee
\label{eq_ecorr_def}
E_{\corr}:=E-E_0,\quad \psi_\corr:=\psi-\psi_0,
\ene
(notice that $P\psi=\psi_\corr$) and rewrite \eqref{eq_rewrite} as
\bee
\label{eq_rewrite_2}
(D-E_0)\psi=(E_\corr-\ep A)\psi.
\ene
Apply $R$ 
to both sides of the rewritten eigenvalue equation \eqref{eq_rewrite_2} (notice that we can do this since the LHS -- and, thus, the RHS as well -- is in the domain of $R$) and arrive at
\begin{equation}
\label{eq_psi_recursion}
\psi=\psi_0+R(E_{\corr}-\ep A)\psi.
\end{equation}
We note that \eqref{eq_psi_recursion} can be considered as a form of the Lippmann -- Schwinger equation \cite{Lippmann}. By projecting \eqref{eq_psi_def} onto $\psi_0$ and using the definition of $E_\corr$, we also have
\bee
\label{eq_ecorr_def2}
E_\corr=\ep \<\psi_0,A\psi\>=\ep \<\psi,A\psi_0\>,
\ene
since $E_\corr\in \R$. 
Noticing that $R(E_\corr-\ep A)\psi_0=-\ep RA\psi_0$, one can iterate \eqref{eq_psi_recursion} and arrive at
\begin{multline}
\label{eq_psi_recursion_multiple}
\psi=\psi_0+\sum_{j=1}^{N-1}\l(R(E_\corr-\ep A)\r)^j\psi_0+\l(R(E_\corr-\ep A)\r)^N\psi=\\=\psi_0-\sum_{j=1}^{N-1}\ep^jR\l\{\<\psi,A\psi_0\>R-AR\r\}^{j-1}A\psi_0+\ep^N\l\{R(\<\psi,A\psi_0\>-A)\r\}^N\psi,
\end{multline}
which leads to 
\begin{multline}
\label{eq_canonical_inner_product}
\<\psi,A\psi_0\>=\\=\<\psi_0,A\psi_0\>-\sum_{j=1}^{N-1}\ep^j\<R\{\<\psi,A\psi_0\>R-AR\}^{j-1}A\psi_0,A\psi_0\>+\ep^N\<\l\{R(\<\psi,A\psi_0\>-A)\r\}^N\psi,A\psi_0\>=\\
=\sum_{j=0}^{N-1}\ep^{j}\l\<\psi_0,\{\<\psi,A\psi_0\>R-AR\}^jA\psi_0\r\>+\ep^{N} \l\<\psi,\{\<\psi,A\psi_0\>R-AR\}^{N}A\psi_0\r\>.
\end{multline}
As a consequence,
from \eqref{eq_ecorr_def2}, we also obtain
\begin{multline}
\label{eq_energy_recursion}
E_\corr=\ep\<\psi,A\psi_0\>=\\
=\sum_{j=0}^{N-1}\ep^{j+1}\l\<\psi_0,\{\<\psi,A\psi_0\>R-AR\}^jA\psi_0\r\>+\ep^{N+1} \l\<\psi,\{\<\psi,A\psi_0\>R-AR\}^{N}A\psi_0\r\>.
\end{multline}
\begin{rem}
\label{remark_domains}
We want to emphasize that both formulas represent exact relations which hold under the general assumptions (A1) -- (A3). One has to be careful with expanding the parentheses in \eqref{eq_psi_recursion_multiple} -- \eqref{eq_energy_recursion}. In particular, \eqref{eq_canonical_inner_product} and \eqref{eq_energy_recursion} only involve monomials in $A$ and $R$ applied to $\psi_0$, and so their expansions are allowed due to (A3). 

In \eqref{eq_psi_recursion_multiple}, the last term $\ep^N\l\{R(\<\psi,A\psi_0\>-A)\r\}^N\psi$ will be referred to as the {\it singular term} and will play a special role. While the term itself is well-defined as a consequence of the eigenvalue equation, the possibility of its expansion is conditional on (A4).
\end{rem}

In the above expressions, the {\it order} of a term is, by definition, the power of $\ep$ in front of that term or, equivalently, the number of times the operator $A$ appears in it.
Each copy of $\<\psi,A\psi_0\>$ in \eqref{eq_energy_recursion} can be further expanded by substituting the expression \eqref{eq_canonical_inner_product}, potentially with a different value of $N$, and this process can also be iterated, thus creating a large variety of possible forms of \eqref{eq_energy_recursion}. In order to introduce the diagrams, it will be convenient to consider a particular sequence of such iterations, with the property that $\psi$ can only appear in the terms of highest order. Such a form is not unique, and our next goal is to choose a particular ``canonical'' sequence of substitutions of the above kind.

We will use the following convention: whenever we ``expand the brackets'', we will always perform this procedure from left to right, and neither change the order of the terms nor the factors in each term, even if the factors are commuting scalars. For example,
\bee
\label{eq_brackers_convention}
(a+b)(c+d)(e+f)=ace+bce+ade+bde+acf+bcf+adf+bdf.
\ene
We now describe the algorithm of obtaining the canonical expansion for \eqref{eq_energy_recursion}. In the beginning, we fix $N$ and declare the right hand side of \eqref{eq_energy_recursion} to be the current expression, which will be transformed as follows.
\begin{itemize}
    \item[(C1)] For $j=1,2,\ldots,N$, expand each factor $\{R\<\psi_0,A\psi\>-RA\}^j$ of the current expression into $2^j$ terms using the convention \eqref{eq_brackers_convention}.
    \item[(C2)] Consider the last (rightmost) term of the current expression of the order not exceeding $N$ and $\<A\psi,\psi_0\>$ appearing at least once. Denote by $i$ the order of that term. Note that, in view of the conventions used in (C1), all occurrences of $\<A\psi,\psi_0\>$ will appear as scalar factors, perhaps located somewhere in the middle of that term.
    \item[(C3)] In the term of order $i$ selected on the previous step, replace the last (rightmost) occurrence of $\<A\psi,\psi_0\>$ by the expansion \eqref{eq_canonical_inner_product} with $N$ replaced by $N+1-i$, which will result in this term being further expanded into $N-i+2$ terms. Note that the order of all new terms will not exceed $N+1$.
    \item[(C4)] Return to step (C1) and repeat.
\end{itemize}
We note that, each time (C3) is applied, an occurrence of $\<\psi,A\psi_0\>$ is replaced by a sum in which every term containing $\psi$ has at least one $\ep$ in front of it. One can say that the appearance of $\psi$ is, therefore, moved into terms of strictly higher order.
Since terms of order $N+1$ are not expanded further, it is easy to see that the procedure will eventually terminate.

A similar algorithm can be applied to the expansion \eqref{eq_psi_recursion_multiple}, with the only exception that the singular term $\ep^N\{R\<\psi,A\psi_0\>-RA\}^N\psi$ stays intact. In this case, the results will be called {\it the canonical eigenvector expansion}. As mentioned earlier, the last term can be expanded if (A4) is additionally imposed.

In both cases, every term will have $\ep^j$ in front, with $1\le j\le N$ for the eigenvector expansions and $1\le j\le N+1$ for the eigenvalue expansions, with $\psi$ only appearing in terms of the highest order (that is, with $\ep$ appearing in the highest possible power).

\subsection{Diagram representations for the canonical expansions} Various terms and factors in the canonical expansions will be represented by certain diagrams which we will call {\it abstract configurations}. Each of these diagrams is a string of symbols, starting with either $\psi$ or $\psi_0$ and ending with $\psi$, $\psi_0$, or (in the eigenvector case) with a circle $\circ$. In between, a diagram may have arrows $\to$, circles $\circ$, parentheses, and additional occurrences of $\psi_0$.

Each abstract configuration will make a certain contribution to the expansion, and we will use notation $\cont(\cdot)$ to indicate such contribution. The simplest abstract configuration, an {\it abstract eigenvalue path}, will be represented by the diagram
\bee
\label{eq_diagram_1}
\psi_0\to\circ\to\ldots\to \circ\to\psi_0=:\psi_0\xrightarrow{\ell}\psi_0,
\ene
where $\ell$ indicates the number of arrows. Sometimes, eigenvalue paths will be referred to as loops, in order to emphasize that, in some sense, they start and end at the same point. Later, we will also consider {\it abstract eigenvector paths} which do not have that property.

The contribution of the diagram \eqref{eq_diagram_1} is, by definition,
$$
\cont(\psi_0\xrightarrow{\ell}\psi_0)=\<\psi_0, A(-RA)^{\ell-1}\psi_0\>.
$$
Each arrow represents an application of $A$ (``edge factor''). Each circle produces an application of $(-R)$ (``vertex factor''). The operators are applied in the order reflected by the direction of the arrows.

It will also be convenient to introduce partial contributions of fragments of the diagrams, for example
$$
\cont(\rightarrow \circ \rightarrow \circ \rightarrow)=A(-R)A(-R)A,\quad \cont(\circ \rightarrow \circ \rightarrow)=(-R)A(-R)A.
$$
In particular, in
$$
\cont(\rightarrow\circ\rightarrow\psi_0)=A(-R)A\psi_0, \quad \cont(\psi_0\rightarrow\circ\rightarrow)=\<\psi_0,A(-R)A\cdot\>,
$$
the first contribution is a vector (``ket vector''), and the second one a linear functional (``bra vector''). Note that the contribution of a diagram is always equal to the product of the contributions of its fragments, whenever the latter are defined.

An abstract eigenvalue path will be a particular case of an {\it abstract eigenvalue configuration}, the central object in the diagram notation, representing most of the terms in the canonical eigenvalue expansion.
\begin{defn}
\label{def_eigenvalue_configuration}
By an {\it abstract eigenvalue configuration}, we mean a result of finitely many {\it attachment procedures} applied, one after another, to an abstract eigenvalue path. Let $\cQ_1,\ldots,\cQ_k$ be abstract eigenvalue paths and $\cP$ an abstract path, or (recursively) an abstract eigenvalue configuration containing a fragment $ \circ\rightarrow$. The attachment procedure is performed as follows:
\bee
\label{eq_attachment_1}
\text{replace}\quad  \circ \rightarrow\quad\text{by}\quad (\cQ_1)(\cQ_2)\ldots(\cQ_k)\rightarrow,
\ene
with the original contribution of the fragment 
\bee
\label{eq_attachment_2}
\cont( \circ\rightarrow)=(-R)A
\ene
replaced by
\begin{multline}
\label{eq_attachment_3}
\cont((\cQ_1)(\cQ_2)\ldots(\cQ_k)\rightarrow)=\cont(\cQ_1)\ldots\cont(\cQ_k)R^{k}(-RA)\\=-\cont(\cQ_1)\ldots\cont(\cQ_k)R^{k+1}A.    
\end{multline}
\end{defn}
\begin{rem}
\label{remark_rephrase}
To rephrase this definition, one can say that each attachment factor $(\cQ_j)$ produces and extra copy of the operator $R$ applied at the location of the attachment, as well as the scalar factor $\cont(\cQ_j)$. Note that the contribution will not depend on the order in which $(\cQ_j)$ are applied at the location of the attachment; however, the corresponding diagrams, for example
$$
\psi_0\to\circ\to(\cQ_1)(\cQ_2)\to \circ\to\psi_0\quad \text{and}\quad \psi_0\to\circ\to(\cQ_2)(\cQ_1)\to \circ\to\psi_0,
$$
with $\cQ_1\neq\cQ_2$, are still counted separately.
\begin{rem}
\label{attachments_same_point}
As defined above, the attachment procedure deletes the original attachment point $\circ$ from the diagram, and therefore no more paths can be attached at the same location. However, since there is no limit on how many paths can be attached during the procedure, this restriction is not important.
\end{rem}

One can also think of attachment process as starting from some abstract eigenvalue path and attaching eigenvalue loops to it, allowing attachments to stack: that is, every point of an attached loop becomes a valid attachment point to add further loops. The original starting path will be often referred to as the {\it base path}. In the case of eigenvector configurations defined later in the text, the base path may have different properties from the attached loops.
\end{rem}

Let $\cP$ be an abstract eigenvalue configuration. Define its {\it length} by
$$
|\cP|:=\text{the number of $A$ factors in $\cont(\cP)$}.
$$
Clearly, $|\psi_0\xrightarrow{\ell}\psi_0|=\ell$, and if $\cP'$ is obtained from $\cP$ by attaching $\cQ$, then $|\cP'|=|\cP|+|\cQ|$.

Abstract eigenvalue configurations are sufficient to represent the terms of the canonical eigenvalue expansion of orders $1\le j\le N$. In the terms of the highest order $N+1$, some occurrences of $\psi_0$ will be replaced by $\psi$, and will require some modifications of the definition of abstract configurations. 

An {\it abstract remainder eigenvalue path} is, by definition, represented by a diagram 
$$
\psi\to\circ\to\ldots\to \circ\to\psi_0=:\psi\xrightarrow{\ell}\psi_0,
$$ 
with $\ell$ arrows, similarly to \eqref{eq_diagram_1}, and
$$
\cont(\psi\xrightarrow{\ell}\psi_0)=\<\psi, A(-RA)^{\ell-1}\psi_0\>.
$$

The attachment rules described in Definition \ref{def_eigenvalue_configuration} are still valid, with some of the paths involved that may be of the remainder kind. However, we need to add an extra attachment rule. In addition to locations $\circ\rightarrow$, the starting point $\psi\rightarrow$ itself also becomes a valid attachment location. For example, suppose that $\cQ_1$ and $\cQ_2$ are abstract eigenvalue configurations. One can attach them to the diagram
$$
\psi\rightarrow\circ\rightarrow\circ\rightarrow\psi_0,
$$
at both circles, as before, as well as at the initial location $\psi\rightarrow$. In the latter case, the new diagram and its contribution are defined by 
$$
\cont(\psi\dashrightarrow(\cQ_1)(\cQ_2)\to \circ\to\circ\to \psi_0):=
\<\psi,\cont(\cQ_1)\cont(\cQ_2)R^2A(-RA)^2\psi_0\>.
$$
Similarly to the usual attachment rules, this attachment also produces an extra $R$ factor, applied last, and a scalar factor $\cont(\cQ)$. In the notation, dashed line does not indicate any contribution and is included solely for aesthetic reasons. We note that, due to self-adjointness, one can think of the last factor $R$ acting on $\psi$. However, without (A4) one cannot assume that $\psi$ is actually in the domain of $R$. We note that one could, in principle, also allow this attachment in the case of the usual loops of the kind $\psi_0\xrightarrow{\ell}\psi_0$; in fact, the corresponding term is even formally present in \eqref{eq_energy_recursion}. In this case, the last $R$ can be moved onto $\psi_0$ and, since $R\psi_0=0$, the contribution of the corresponding diagram becomes zero. As above, the attachments can be iterated, with both abstract eigenvalue paths and remainder eigenvalue paths involved.

The actual definition of an abstract remainder eigenvalue configuration imposes several further restrictions on top of the above formal rules, which can be traced back to the conventions (C1) -- (C4) that were used in the definitions of the canonical expansions.
\begin{defn}
\label{def_remainder_eigenvalue_configuration}
An {\it abstract remainder eigenvalue configuration} is a result of finitely many attachments described above, with the following properties.
\begin{itemize}
	\item[(R1)] There is at least one abstract remainder  path involved.
	\item[(R2)] There is at most one abstract remainder eigenvalue path that has attachments.
	\item[(R3)] Every path attached to a abstract remainder eigenvalue path must be a remainder eigenvalue path of length $1$.
	\item[(R4)] Every path located on the diagram to the left of some abstract remainder eigenvalue path, must also be a remainder eigenvalue path of length $1$.
\end{itemize}
\end{defn}
In order to treat the canonical eigenvector expansion \eqref{eq_psi_recursion_multiple}, we will need to introduce eigenvector paths. Compared to the eigenvalue case, the situation complicates somewhat due to the fact that the expansion of the remainder term $\ep^N\l\{R(\<\psi,A\psi_0\>-A)\r\}^N$ in \eqref{eq_psi_recursion_multiple} relies on whether (A4) is assumed to hold.
\begin{defn}
\label{def_abstract_eigenvector_path}
An {\it abstract eigenvector path} is represented by the diagram 
$$
\circ\xrightarrow{}\ldots\xrightarrow{}\circ\rightarrow\psi_0=\circ\xrightarrow{\ell}\psi_0,
$$
with $\ell$ arrows and
$$
\cont(\circ\xrightarrow{\ell}\psi_0)=(-RA)^{\ell}\psi_0.
$$
\end{defn}
\begin{defn}
An {\it abstract eigenvector configuration} is the result of finitely many attachments described in Definition \ref{def_eigenvalue_configuration} to an abstract eigenvector path, with contribution rules described in \eqref{eq_attachment_1} -- \eqref{eq_attachment_3}. Note that all attachment must be of the eigenvalue kind.
\end{defn}
\begin{defn}
An {\it abstract remainder eigenvector configuration} is the result of finitely many attachments described in Definition \ref{def_remainder_eigenvalue_configuration} to an abstract eigenvector path, following the rules (R1) -- (R4). Note that the original base path is of the same kind as in the previous definition, but the attached eigenvalue loops may be of either kind, within the above constraints.
\end{defn}
Suppose that (A4) holds. Then, one can additionally define the associated diagrams.
\begin{defn}
A {\it singular eigenvector path} is represented by the diagram 
$$
\circ\xrightarrow{}\ldots\xrightarrow{}\circ\rightarrow\psi=\circ\xrightarrow{\ell}\psi,
$$
with $\ell$ arrows and
$$
\cont(\circ\xrightarrow{\ell}\psi)=(-RA)^{\ell}\psi.
$$
\end{defn}

Finally, a singular eigenvector configuration is the kind of a remainder configuration that appears after expanding the singular term $\ep^N\l\{R(\<\psi,A\psi_0\>-A)\r\}^N\psi$.
\begin{defn}
A {\it singular eigenvector configuration} is the result of finitely many attachments of abstract eigenvector paths $\psi\to\psi_0$ to a singular eigenvector path.
\end{defn}
\begin{rem}
We note that the singular configurations are defined in line with the rules (R1) -- (R4). Indeed, once the base path is of singular kind, the only attachments allowed must be remainder eigenvalue paths of length $1$. In the case of singular paths, attachments can also be made at the last location $\psi$. For example, attachment of three copies of $\psi\to \psi_0$ at the location $\psi$ is represented by the following transformation:
$$
\text{replace}\quad \ldots\psi\quad\text{by}\quad \ldots(\psi\to \psi_0)(\psi\to \psi_0)(\psi\to \psi_0)\dashrightarrow \psi.
$$
Each copy of $(\psi\to \psi_0)$ contributes a factor $\<\psi,A\psi_0\>R$ applied to $\psi$.
\end{rem}
\begin{thm}
\label{th_canonical_expansion}
Assume that $D$, $E$, and $\psi$ satisfy $\mathrm{(A1)}$ -- $\mathrm{(A3)}$, and define $E_\corr$ as in \eqref{eq_ecorr_def}. Then, for every $N\in \N$, we have
\begin{multline}
\label{eq_ecorr_th}
E_\corr=\\=\sum\{\ep^{|\cP|}\cont(\cP)\colon \text{$\cP$ is an abstract eigenvalue configuration with $1\le |\cP|\le N$}\}+\\
+\sum\{\ep^{N+1}\cont(\cP)\colon \text{$\cP$ is an abstract remainder eigenvalue configuration with $|\cP|=N+1$}\}.
\end{multline}
\begin{multline}
\label{eq_psicorr_th}
\psi_\corr=\\=\sum\{\ep^{|\cP|}\cont(\cP)\colon \text{$\cP$ is an abstract eigenvector  configuration with $1\le |\cP|\le N-1$}\}+\\
+\sum\{\ep^N\cont(\cP)\colon \text{$\cP$ is an abstract remainder eigenvector configuration with $|\cP|=N$}\}+\\+\ep^N\l\{R(\<\psi,A\psi_0\>-A)\r\}^N.
\end{multline}
If, in addition, $\mathrm{(A4)}$ is satisfied, then \eqref{eq_psicorr_th} can be further expanded
\begin{multline}
\label{eq_psicorr_th_a4}
\l\{R(\<\psi,A\psi_0\>-A)\r\}^N=\\=\sum\{\cont(\cP)\colon \text{$\cP$ is an singular abstract remainder eigenvector configuration with $|\cP|=N$}\}.
\end{multline}
\end{thm}
\begin{proof}
The equality \eqref{eq_ecorr_th} follows from \eqref{eq_ecorr_def2} and \eqref{eq_psicorr_th}. In order to prove the latter, it is sufficient to establish a one-to-one correspondence between the terms obtained as a result of the canonical eigenvector expansion and the abstract configurations involved in the \eqref{eq_psicorr_th}.

For each configuration $\cP$ among those considered in \eqref{eq_psicorr_th}, define its {\it base length} as the sum of the length of the base loop of $\cP$ and the number of attachments that are made directly to the base loop. Clearly, the base length can be between $1$ and $N+1$ for the case of eigenvector configurations, and the configurations can be classified by their base length and the positions of attached paths.

In the language of the terms of \eqref{eq_psi_recursion_multiple}, the index $j$ can be identified with the base length of the corresponding paths. After expanding the corresponding factor $\{\<\psi_0,A\psi\>R-AR\}^j$ into $2^j$ terms (without changing the order in which the factors $\<\psi_0,A\psi\>R$ and $(-RA)$ are considered), each term will represent the number and locations of attachments. More precisely, the factor $(-RA)$ means one step being added to the base loop, and $R\<\psi_0,A\psi\>$ means an attachment has been made. Afterwards, we can proceed by induction in $N$ and noting that, by induction assumptions, the expansion of $R\<\psi_0,A\psi\>$ will include all allowed eigenvalue configurations. 

One has to be somewhat careful regarding the remainder terms; however, the reader can easily check that the rules (R1) -- (R4) of attachment of the remainder paths reflect the order in which the recursion in the canonical expansion (C1) -- (C4) is performed. In particular, the special role of paths of length 1 in (R3) and (R4) is due to the fact that the factors $\<A\psi_0,\psi\>$, corresponding to such paths, are already present in \eqref{eq_psi_recursion_multiple} and \eqref{eq_energy_recursion} before any canonical expansion. The rule (R3) corresponds to the fact that any other way for $\psi$ to appear would be in the term of highest order produced by the inductive step (C3). However, this means that the expansion has not yet reached preceding factors in that term, and those factors cannot be expanded further without increasing the order.

Finally, if (A4) holds, the expansion of the last term $\ep^N\l\{R(\<\psi,A\psi_0\>-A)\r\}^N$ only involves expanding the brackets, since any substitution will push the order of the terms above $N$. The reader can easily check that the terms obtained as the result of such expansion correspond exactly to singular configurations considered in \eqref{eq_psicorr_th_a4}.
\end{proof}
\subsection{Applications to operators with pure point spectra}
In this section, we will additionally assume that the operator $D$ has pure point spectrum: that is, without loss of generality, $D$ acts on $\ell^2(\mathbb I)$ with the standard basis $\{e_{\bn}\colon\bn\in\mathbb I\}$ by
\bee
\label{eq_d_diagonal_def}
D{e_\bn}=d_\bn e_{\bn},\quad \bn\in \mathbb I,
\ene
We note that the most common application here is the case $\mathbb I=\Z^d$, with $D$ being the potential and $A$ being the hopping term of a Schr\"odinger-type operator
$$
H=D+\ep A.
$$
In this setting, diagram representations of the terms of RSS were used in \cite{KPS} in order to establish convergence of perturbation series (see also recent work \cite{faust_k,velocity} that uses perturbation series and diagrams for isolated Floquet eigenvalues in the periodic case). The goal of the this subsection is to rewrite the results of Subsection \ref{subsection_canonical_expansions} in terms of the matrix elements of the operators $D$ and $A$, introducing expressions for the remainder configurations while mostly preserving the notation of \cite{KPS}.

\begin{defn}
\label{def_finite_range}
We say that $A$ has {\it finite range with respect to $D$} if, for every $\bn\in \I$, $Ae_\bn$ is a finite linear combination of the standard basis vectors of $\ell^2(\I)$.
\end{defn}
Assume now
\bee
\label{eq_non_resonant}
\psi_0=e_{\bm},\quad E_0=d_\bm,\quad d_\bn\neq d_\bm\quad\text{for}\quad\bn\neq \bm.
\ene
Suppose that
\bee
\label{eq_psi_lattice}
(D+\varepsilon A)\psi=E\psi,\quad \<\psi,e_\bm\>=1.
\ene
In other words, we consider $d_\bm$ as the unperturbed eigenvalue and $e_\bm$ as the unperturbed eigenvector, therefore satisfying (A1) and (A2) from Subsection \ref{subsection_canonical_expansions}. We will also assume that $A$ has finite range with respect to $D$, which, together with the earlier assumptions, implies (A3). 
\begin{rem}
\label{rem_a3}
Definition \ref{def_finite_range} is not the only natural way to impose (A3). Another important sufficient condition in the case $\I=\Z^d$ is both $A$ and $R$ being {\it localization preserving}, see Definition \ref{def_localization_preserving}.
\end{rem}
Let
\bee
\label{eq_ecorr_psicorr_lattice}
E_\corr:=E-d_\bm,\quad \psi_\corr=\psi-e_\bm.
\ene
Similarly to \cite{KPS}, we now introduce the language of loops in the pure point setting. A {\it path} is, by definition, represented by a diagram of the form
\begin{equation}
\label{eq_p_def}
\cP=\bn_0\to \bn_1\to \ldots\to \bn_k=\bn\to \bn_1\to \ldots\to\bn_{k-1}\to \bm,
\end{equation}
where 
$$
\bn_k=\bm \in \I,\quad \bn_0=\bn\in \I,\quad \bn_1,\ldots\bn_{k-1}\in \I\setminus\{\bn_0\}.
$$
If $\bm=\bn$, the corresponding path will be called an {\it eigenvalue path} or a {\it loop}. If $\bm\neq \bn$, it will be called an {\it eigenvector path}. Let $\cP$ be a path, $\cQ=\bm\to \bm_1\to \ldots\bm_{\ell-1}\to \bm$ be an eigenvalue path, and $\bn_j\neq \bm$ be one of the entries in $\cP$. The {\it attachment procedure} (at the position $j$) is performed by replacing $\bn_j$ in \eqref{eq_p_def} by $\bn_j\to (\cQ)\to \bn_j$, resulting in an object represented by the diagram
\begin{equation}
\label{eq_attachment_example}
\bn\to \bn_1\to \ldots\to \bn_j\to (\bm\to\bm_1\to\ldots\to \bm_{\ell-1}\to \bm)\to \bn_j\to\ldots\to \bn_{k-1}\to \bm.
\end{equation}
Note that the condition $\bn_j\neq \bm$ means that one cannot attach anything at the first and last point of an eigenvalue path and at the last point of an eigenvector path. The attachments can be iterated. In \eqref{eq_attachment_example}, the valid attachment points are $\bn_1,\ldots,\bn_{\k-1}$, $\bm_1,\ldots\bm_{\ell-1}$, and $\bn$ if $\bn\neq \bm$ (note that one always has $\bn_k=\bm$. The points $\bn_j$ appear twice, meaning that further loops may be attached either ``before'' or ''after'' $\cQ$. Unlike the notation in \eqref{eq_attachment_1} -- \eqref{eq_attachment_3}, the extra $\bn_j$ reflected in the notation on both sides, in order to indicate the specific element of $\I$ at which the attachment happens.

Depending on the type of the original path $\cP$, we call the result of adding attachments to it an {\it eigenvalue or eigenvector configuration}. Similarly to the previous section, all attached paths must be of the eigenvalue kind.

For a path $\cP$ introduced in \eqref{eq_p_def}, its contribution is, by definition,
$$
\cont(\cP):=\begin{cases}
\prod_{j=0}^{k-1}{A_{\bn_{j}\bn_{j+1}}}(d_{\bm}-d_{\bn_j})^{-1},&\bn\neq \bm;\\
A_{\bn_0\bn_1}\prod_{j=1}^{k-1}A_{\bn_{j}\bn_{j+1}}(d_{\bm}-d_{\bn_j})^{-1},&\bn=\bm.
\end{cases}
$$
For the attachments, it is defined inductively: if
$$
\cP=\cP_1\to \bk\to \cP_2,\quad \cP'=\cP_1\to \bk\to (\cQ)\dashrightarrow \bk\to \cP_2,
$$
where $\cP$ is a configuration and $\cQ$ is an eigenvalue path, then, by definition
$$
\cont(\cP'):=-(d_{\bm}-d_{\bk})^{-1}\cont(\cP)\cont(\cQ).
$$
One can interpret the contribution of each configuration as the product of ``vertex factors'', ``edge factors'', and ``attachment factors'' that are included in that configuration, using the following rules:
\begin{enumerate}
    \item Each $\bk\in \I\setminus\{\bm\}$ that has an outbound arrow to another point  (e. g. $\bk\to \bk'$) produces a vertex factor $(d_\bm-d_\bk)^{-1}$.
    \item Each arrow $\bk'\to \bk$ produces an edge factor $A_{\bk'\bk}=\< Ae_{\bk'},e_{\bk}\>$.
    \item Each point $\bk\in\I$ which has an outbound arrow towards an attachment (e. g. $\bk\to (...)$) produces an attachment factor $-(d_{\bm}-d_{\bk})^{-1}$. Note that the rules of attachment require $\bk\neq \bm$.    
\end{enumerate}
The arrow $(...)\dashrightarrow \bk$ immediately following the attachment does not produce any contribution, which is reflected in the dashed notation.

Let
$$
|\cP|:=\{\text{the number of edge factors in }\cont(\cP)\}.
$$
\begin{rem}
\label{rem_differences_kps}
Unlike \cite{KPS}, the operator $A$ can have diagonal entries, and therefore some diagrams may contain arrows $\bk\to \bk$. In particular, $\bm\to \bm$ is a valid eigenvalue path with $|\bm\to \bm|=1$ and $\cont(\bm\to \bm)=\<Ae_\bm,e_\bm\>$.

The reader may also notice that the direction of the arrows has changed to the opposite compared to \cite{KPS}. In most cases, it does not affect the results due to the fact that the operators involved are self-adjoint. The current choice is somewhat dictated by the choice of order in Subsection \ref{subsection_canonical_expansions}, which follows the usual order in which composition of operators is written.
\end{rem}
In order to describe the remainder terms of highest orders, let $\psi\in \ell^2(\Z^d)$ and consider {\it remainder eigenvalue paths} denoted by
\begin{equation}
\label{eq_psipath_def}
\psi\to \bn_1\to\ldots\to \bn_{\ell-1}\to \bm.
\end{equation}
By definition,
\begin{equation}
\label{eq_psipath_cont}
\cont(\psi\to \bn_1\to\ldots\to \bn_{\ell-1}\to \bm):=\<A\psi,e_{\bn_1}\>\prod_{j=1}^{\ell-1}\<A e_{\bn_{j}},e_{\bn_{j+1}}\>(d_{\bm}-d_{\bn_{j}})^{-1}
\end{equation}
In other words, the first part $\bm\to \bn_{1}$ is replaced by $\psi \to \bn_1$, and the contribution of the corresponding arrow is now $\<A\psi,e_{\bn_1}\>$ rather than $A_{\bn_{0}\,\bn_{1}}=\<A e_{\bm},e_{\bn_{1}}\>$.

In the above diagram, the rules of attachment to $\bn_1,\bn_2,\ldots, \bn_{\ell-1}$ are the same as above. In the case of remainder paths, we are also required to introduce attachments to $\psi$. Recall that, in the language of the previous section, attaching a path $\cQ$ at that location amounts to applying an extra copy of $R$ and adding a scalar factor $-\cont(\cQ)$. 
One can also formally move the action of $R$ onto $\psi$ and write
$$
R\psi=\sum_{\bk\colon\bk\neq\bm} (d_\bm-d_\bk)^{-1} \psi_\bk e_\bk,\quad \text{where}\quad \psi=\sum_{\bk\in \Z^d} \psi_\bk e_{\bk}.
$$
Such representation would require some version of (A4) to hold. However, if one considers diagram representations for individual matrix elements, this problem does not appear, and we arrive to the diagram
\begin{multline*}
\cont(\psi\to\bk\to (\cQ)\dashrightarrow \bk\to \bn_1\to\ldots\to \bn_{\ell-1}\to \bm):=\\:=-\<\psi,e_\bk\>(d_{\bm}-d_{\bk})^{-1}\<Ae_\bk,e_{\bn_1}\>\prod_{j=1}^{\ell-1}\<A e_{\bn_{j}},e_{\bn_{j+1}}\>(d_{\bm}-d_{\bn_{j}})^{-1}.	
\end{multline*}
More formally, if 
$$
\cP=\psi\to \bn_1\to \cP', \quad \cont(\cP)=\<A\psi,e_{\bn_1}\>\cdot C,
$$
one can produce an attachment of $\cQ$ to $\psi$ at the lattice point $\bk$, represented by the diagram with the following contribution
$$
\cont(\psi\to \bk\to(\cQ)\dashrightarrow \bn_1\to \cP')=-\<\psi,e_{\bk}\>(d_{\bm}-d_{\bk})^{-1}\<Ae_{\bk},e_{\bn_1}\>\,\cont(\cQ)\cdot C.
$$
\begin{rem}
\label{rem_no_a4}
One can also view the above construction as a way to expand the singular term on the level of individual matrix elements: we have
$$
\<e_{\bn},\l\{R(\<\psi,A\psi_0\>-A)\r\}^N\psi\>=\<\l\{R(\<\psi,A\psi_0\>-A)\r\}^N e_{\bn},\psi\>,
$$
and the brackets can now be expanded in the left hand side due to $A$ being of finite range with respect to $D$, without assuming (A4).
\end{rem}

For the remainder eigenvalue and eigenvector configurations, the general rules are the same as in the previous section, and we restate them as follows.
\begin{itemize}
	\item There is at least one remainder path involved.
	\item There is at most one remainder path that has attachments.
	\item Every loop attached to a remainder must be a remainder loop of length $1$.
	\item Every loop located on the diagram to the left of some remainder path, must also be a remainder loop of length $1$.
\end{itemize}
We can now state the main theorem result of this section, which is a restatement of Theorem \ref{th_canonical_expansion} in the language of eigenvalues of $D$ and matrix elements of $A$.
\begin{thm}
\label{th_lattice_diagram_expansion}
Assume that $D$, $E$, $\psi$, $E_\corr$, and $\psi_\corr$ are defined as in \eqref{eq_d_diagonal_def}, \eqref{eq_psi_lattice}, \eqref{eq_ecorr_psicorr_lattice}, and that $A$ has finite range with respect to $D$. Then, for every $N\in \N$, we have
\begin{multline*}
E_\corr=\sum\{\ep^{|\cP|}\cont(\cP)\colon \text{$\cP$ is an eigenvalue configuration with $1\le |\cP|\le N$}\}+\\
+\sum\{\ep^{N+1}\cont(\cP)\colon \text{$\cP$ is an remainder eigenvalue configuration with $|\cP|=N+1$}\};
\end{multline*}
\begin{multline*}
\psi(\bn)=e_{\bm}+\sum\{\ep^{|\cP|}\cont(\cP)\colon 1\le |\cP|\le N,\,\,\text{$\cP$ is an eigenvector configuration}\}+\\+
	\sum\{\ep^N\cont(\cP)\colon |\cP|=N+1,\,\,\text{$\cP$ is a remainder eigenvector configuration}\}.
\end{multline*}
\end{thm}
\begin{proof}
The proof follows by inspection of that of Theorem \ref{th_canonical_expansion} in the language of matrix elements. In the standard basis, the operator $R$ is diagonal:
$$
R e_\bn=\begin{cases}
(d_\bm-d_\bn)^{-1}e_\bn,& \bn\neq\bm;\\
0,&\bn=\bm,
\end{cases}
$$
which provides the vertex and attachment factors. The form of the edge factors follows from the matrix multiplication rules expressed in the standard basis of $\ell^2(\I)$. In view of Remark \ref{rem_no_a4}, the singular term no longer requires special treatment.
\end{proof}
\begin{rem}
\label{rem_localization_preserving_2}
As mentioned earlier, the assumption of $A$ being of finite range can be replaced by both $A$ and $R$ being localization preserving, see Definition \ref{def_localization_preserving}.
\end{rem}
\begin{rem}
\label{remark_arbitrary_eigenfunction}
 It may seem somewhat counterintuitive that the expansions constructed above hold for every eigenvalue/eigenvector pair $(E,\psi)$ of $D+\ep A$ and every choice of the unperturbed vector $e_\bm$, as long as $\psi$ can be normalized in order to satisfy $\<\psi,e_{\bm}\>=1$. Moreover, the first $N$ terms of the expansion do not depend on $\psi$, which means that the remainder may be very large. Note that, in the case when $\psi$ represents an eigenvalue branch different from the one arising from $e_{\bm}$, the orthogonality condition in the limit $\ep\to 0$ combined with our choice of normalization would imply $\|\psi \|$ blowing up.

\end{rem}
\section{Differentiability and first order calculations}
In this section, we discuss two somewhat general consequences of the expansions from the previous section. The first one, Theorem \ref{th_smoothness}, states that if $D+tA$ is an operator on $\ell^2(\Z^d)$ with an eigenvector family exponentially localized in $t$ (without additional regularity assumptions besides an appropriate normalization), then the family is forced to be infinitely smooth in $t$, under some general conditions.

In the second subsection, we discuss in more detail the case of a first order expansion (with the remainder estimate of second order), as well as the particular case of a rank one perturbation, which will be useful in  more concrete applications of the later sections.
\subsection{Smoothness}
\label{subsection_smoothness}In order to discuss smoothness, we start from introducing several convenient concepts related to exponential localization.
\begin{defn}
We will say that $\psi\in \ell^2(\Z^d)$ is $(C,\gamma)$-{\it localized} at $\bn_0$ if 
\bee
\label{eq_psi_localized}
|\psi(\bn-\bn_0)|\le Ce^{-\gamma|\bn-\bn_0|},\quad\forall \bn\in \Z^d.
\ene
A family of vectors is {\it uniformly localized at $\bn_0$} if \eqref{eq_psi_localized} holds with some choice of $C,\gamma$ uniform in the family. In the case of $\bn_0$ being the origin, we will drop it from the notation and will just use the terminology {\it $(C,\gamma)$ $($uniformly$)$ localized}.
\end{defn}
\begin{defn}
\label{def_localization_preserving}
We will say that a (possibly unbounded) linear self-adjoint operator $R$ on $\ell^2(\mathbb Z^d)$ is {\it localization preserving} if for every $C,\gamma>0$ there exists $C',\gamma'>0$ such that $R\psi$ is defined and $(C',\gamma')$-localized whenever $\psi$ is $(C,\gamma)$-localized.    
\end{defn}

\begin{prop}
\label{prop_localization_preserving}
Under the above assumptions, let $Q_M:=\one_{\{\bn\in \Z^d\colon |\bn|>M\}}$. Then,
$$
\|R Q_M\psi\|\to 0,\quad \text{as}\quad M\to +\infty
$$
for every localized vector $\psi$ and every localization preserving operator $R$, and the convergence is uniform in $\psi$ over any uniformly localized family of vectors $\psi$.
\end{prop}
\begin{proof}
Let
$$
\|\psi\|_{2,\gamma}^2:=\sum_{\bn\in \Z^d}(1+e^{\gamma|\bn|})^2|\psi(\bn)|^2,
$$
and let $\ell^2_\gamma(\Z^d)$ be the $\ell^2$ space with the above norm. Perhaps after a modification of $\gamma'$, every localization preserving operator becomes bounded if one considers it between $\ell^2_\gamma(\Z^d)$ and $\ell^2_{\gamma'}(\Z^d)$. Then, the conclusion follows from observing
$$
\|RQ_M\psi\|_2\le C(\gamma')\|RQ_M\psi\|_{2,\gamma'}\le C'(R,\gamma')\|Q_M\psi\|_{2,\gamma}\to 0,\quad \text{as}\quad M\to +\infty.\,\qedhere
$$
\end{proof}
We will now discuss two elementary examples of localization preserving operators, with the proof left to the reader.
\begin{prop}
\label{prop_localization_preserving_examples}
\begin{enumerate}
	\item A diagonal operator $R e_\bn=r_\bn e_\bn$ with polynomially bounded entries 
$$
|r_\bn|\le C(1+|\bn|)^{\tau}
$$
is localization preserving.
\item An operator with uniformly exponentially decaying matrix elements
$$
|\<R e_\bm,e_\bn\>|\le C e^{-\gamma|\bn-\bm|}
$$
is localization preserving.
\end{enumerate}
\end{prop}
\begin{thm}
\label{th_smoothness}
Under the assumptions and in the notation of Theorem $\ref{th_canonical_expansion}$, suppose that $R$  and $A$ are localization preserving on $\ell^2(\Z^d)$. Assume that $\psi(t)$ is a family of uniformly localized eigenvectors:
\bee
\label{eq_smoothness_1}
H(t)=(D+tA)\psi(t)=E(t)\psi(t),\quad \<\psi,\psi_0\>=1,\quad \forall t\in I,
\ene
where $0\in I\subset \R$ is an open interval and $\psi(0)=\psi_0$. Then, both $\psi$ and $E$ are infinitely differentiable at $0$. Moreover, the derivatives in $t$ at $0$ are given by the terms of the perturbation series, and the remainders in the Taylor formulas for $\psi$ are uniformly localized, with $C',\gamma'$ depending on the order of the derivative.
\end{thm}
\begin{proof}
Theorem \ref{th_canonical_expansion} provides a Taylor expansion of $\psi(t)$ and $E(t)$ near the origin. Note that the remainder is uniformly bounded in $t$ since it involves finitely many repeated applications of $R$ and $A$ to vectors which are known to be uniformly exponentially localized.
\end{proof}
\begin{rem}
\label{rem_one_sided}
Theorem \ref{th_smoothness} can easily be extended to the case of half-open intervals and one-sided derivatives at the corresponding endpoints.
\end{rem}
\begin{cor}
\label{cor_uniform_smoothness}
Under the assumptions and in the notation of Theorem \ref{th_canonical_expansion}, suppose that $R$ and $A$ are localization preserving on $\ell^2(\Z^d)$. Assume that $\psi(t)$ is a family of uniformly localized eigenvectors, with each eigenvalue $E(t)$ being simple:
\bee
\label{eq_smoothness_2}
H(t)=(D+tA)\psi(t)=E(t)\psi(t),\quad \forall t\in I.
\ene
In addition, assume that for every $t\in I$ there are $\delta(t),c(t)>0$ such that
\bee
\label{eq_branch_switch}
|\<\psi(s),\psi(t)\>|\ge c(t),\quad \forall s\in (t-\delta(t),t+\delta(t)).
\ene
Then, $E(\cdot)$ is infinitely smooth on $I$, and $\psi(\cdot)$ can be normalized to be smooth on $I$.
\end{cor}
\begin{proof}
The smoothness of $E$ follows from Theorem \ref{th_smoothness} applied at every particular $t$. In order to obtain smoothness of $\psi$, first observe the following: if it is already known that a smooth normalization exists, then, in order to verify smoothness of an arbitrary normalization, it is sufficient to check smoothness of any of its non-zero components. More precisely, fix some $\varphi \in \ell^2$ with $\<\varphi,\psi\>\neq 0$. Suppose that there exists a normalization of $\psi$ that is smooth at $s$. Since any two normalizations differ by a scalar factor, an arbitrary normalization $\widetilde \psi$ will be smooth at $s$ if and only if $\<\varphi,\widetilde\psi(\cdot)\>$ is smooth at $s$. 

Now, fix some $t\in I$ and denote
\bee
\label{eq_normalization}
\widetilde\psi(s):=\frac{\psi(s)}{\<\psi(s),\psi(t)\>},\quad s\in (t-\delta(t),t+\delta(t)).
\ene
The above normalization satisfies the assumptions of Theorem $\ref{th_smoothness}$ at $s=t$ and is therefore differentiable at $s=t$. As a consequence, for every particular $t$ there exists a normalization smooth at $t$. On the other hand, for a fixed $t$ the normalization \eqref{eq_normalization} satisfies $\<\widetilde\psi(s),\varphi\>\equiv 1$ with $\varphi=\psi(t)$. In view of the earlier observation, \eqref{eq_normalization} must be smooth for all $s$ within the range. As a consequence, a smooth parametrization exists locally near every point $t\in I$, and the general result follows from local compactness of $I$.
\end{proof}
\subsection{The explicit calculations in the first order case}\label{subsection_first_order}
For the most important applications, we will require some more explicit calculations regarding the first order perturbation and error estimates. One can consider it as a form of Hellmann--Feynman theorem with an explicit remainder without an a priori differentiability assumption. 

We will perform the calculations in the operator language of Subsection \ref{subsection_canonical_expansions}. From \eqref{eq_psi_recursion_multiple}, we have for $N=2$:
\begin{multline}
\label{eq_psi_first_order}
\psi=\psi_0-\ep RA\psi_0+\ep^2\{R\<\psi,A\psi_0\>-RA\}^2\psi=\\=(1-\ep RA)\psi_0+\ep^2\{R\<\psi,A\psi_0\>-RA\}^2\psi=\psi_0+\ep\l(\psi_0\xrightarrow{} \psi_0\r)+\\+\ep^2\l((\psi\xrightarrow{}\psi_0)(\psi\xrightarrow{}\psi_0)\dashrightarrow\psi+\circ\xrightarrow{}(\psi\xrightarrow{}\psi_0)\dashrightarrow{}\psi+
(\psi\xrightarrow{}\psi_0)\xrightarrow{}\psi+\circ\xrightarrow{2}\psi\r).
\end{multline}
Some of the terms in the above expression were removed due to the fact that $R\psi_0=0$. Note that one automatically gets the second order expansion by replacing all appearances of $\psi$ in the last term by $\psi_0$ (which will lead to vanishing of the first two terms of order $\ep^2$ by the same reasons).

One also has
\begin{multline}
\label{eq_E_first_order}
E_{\corr}=\ep \<\psi,A\psi_0\>=\ep \<\psi_0,A\psi_0\>+\ep^2 \<\psi, \l\{\<\psi,A\psi_0\>R-AR\r\}\psi_0\>=\\=\ep(\psi_0\to \psi_0)+\ep^2\l(\psi\dashrightarrow(\psi\to\psi_0)\to \psi_0+\psi\xrightarrow{2}\psi_0\r).
\end{multline}
Suppose now that
$$
A=\<\varphi,\cdot\>\varphi,\quad\|\varphi\|=1,
$$
is of rank one. In this case, we have
\begin{multline}
\label{eq_psi_rank_one}
\psi=\psi_0-\ep \<\varphi,\psi_0\>R\varphi+\\
+\ep^2 \<\varphi,\psi\>\{\<\varphi,\psi_0\>^2\<\varphi,\psi\> R^2\psi -\ep^2\<\varphi,\psi_0\>\<\varphi,\psi\>R^2\varphi-\<\varphi,\psi_0\>\<R\psi,\varphi\>R\varphi+\<R\varphi,\varphi\>R\varphi\}.
\end{multline}
\section{Quasiperiodic operators with monotone potentials: smoothness and gap labeling}
\subsection{The setting}
In the remaining sections, we will apply the results of Sections 1 -- 3 to multi-dimensional quasiperiodic operators with monotone potentials, in the perturbative localization regime considered in \cite{KPS3}. Let
\bee
\label{eq_h_def}
(H\psi)(\bn)=\varepsilon(\Delta\psi)(\bn)+V(\bn)\psi(\bn),\quad V(\bn)=f(x+\bn\cdot\omega).
\ene
While the following setting is not the most general possible, we will restrict ourselves to the following class of {\it Lipschitz monotone} functions. Let $f\colon \R\to [-\infty,+\infty)$ be $1$-periodic and satisfying
\bee
\label{eq_monotonicity}
f(y)-f(x)\ge (y-x),\quad \text{for}\quad 0\le x\le y<1.
\ene
Fix $C_{\dio},\tau_{\dio}>0$ and assume that $\omega\in \R^d$ is a Diophantine frequency vector:
\bee
\label{eq_dio_def}
\omega\in \mathrm{DC}:=\l\{\omega\in [0,1)^d\colon \|\bn\cdot\omega\|:=\dist(\bn\cdot\omega,\Z)\ge C_\dio|\bn|^{-\tau_\dio} ,\bn\neq \bze\r\}.
\ene
Let also $T^{\bn}$ be the translation operator
$$
T^\bn(\psi)(\bm):=\psi(\bm+\bn).
$$
In \cite{KPS3}, it is shown that the the family \eqref{eq_h_def} satisfies uniform localization for small $\varepsilon$ depending on $\omega$. More precisely, the following is the main result of \cite{KPS3} in the above case:
\begin{prop}
\label{prop_main_kps3} Let $\delta>0$. There exists $\ep_0=\ep_0(d,C_{\dio},\tau_{\dio},\delta)>0$ such that for every $\ep\in (0,\ep_0)$, $\omega\in \mathrm{DC}$, and $f$ satisfying \eqref{eq_monotonicity}, one can find a $1$-periodic function $E\colon \R\to[-\infty,+\infty)$ also satisfying \eqref{eq_monotonicity}, and a $1$-periodic measurable function $\psi\colon [0,1)\to \ell^2(\Z^d)$ such that
\bee
\label{eq_sinai_eigenfunction}
H(x)\psi(x)=E(x)\psi(x),\quad\forall x\in \R;
\ene
\bee
\label{eq_uniformly_localized}
\|\psi(x)\|_{\ell^2(\Z^d)}=1,\quad |\psi(x,\bze)-1|<\ep^{1-\delta};\quad |\psi(x;\bn)|\le \ep^{(1-\delta)|\bn|}\quad \text{for}\quad |\bn|\neq 0.
\ene
As a consequence, for every $x\in \R$ one has
\bee
\label{eq_covariance}
H(x)(T^{\bn}\psi(x-\bn\cdot\omega))=E(x-\bn\cdot\omega)(T^{\bn}\psi(x-\bn\cdot\omega)).
\ene
The corresponding family
\bee
\label{eq_eigenvectors}
\psi_{\bn}(x):=T^{\bn}\psi(x-\bn\cdot\omega),\quad \bn\in\Z^d,
\ene
is an orthonormal basis of eigenvectors of $H(x)$, with eigenvalues 
\bee
\label{eq_eigenvalues}
E_{\bn}(x):=E(x-\bn\cdot\omega),
\ene
respectively. In particular, the spectrum of $H(x)$ is pure point and simple,  with uniformly exponentially localized eigenfunctions.
\end{prop}
\begin{rem}
\label{rem_notation}
In the notation $\psi_\bn(x)$, $\bn\in\Z^d$ is the label of the eigenfunction (the eigenfunction with label $\bn$ will always be localized at $\bn\in \Z^d$), and $x$ is the phase. The value of the eigenfunction $\psi_\bn(x)$ at the point $\bm\in \Z^d$ will be denoted by
$$
\psi_\bn(x;\bm):=\<\psi_\bn(x),e_\bm\>.
$$
\end{rem}
\begin{rem}
\label{rem_monotonicity_constant}
By rescaling, one can always replace \eqref{eq_monotonicity} by
\bee
\label{eq_rem43}
f(y)-f(x)\ge c(y-x),\quad \text{for}\quad 0\le x\le y<1,
\ene
and the estimates can be made uniform in $c$ assuming that the range of its values is a compact subset of $(0,+\infty)$.
\end{rem}
\subsection{Smoothness of eigenvalues and eigenvectors}
The following lemma serves as a preparation for applying perturbation theory.
\begin{lem}
\label{lemma_localization_preserving}
Under the assumptions of Proposition $\ref{prop_main_kps3}$, for every $\bn\in \Z^d$, the projected resolvent
$$
R_{\bn}(x):=(H(x)-E_\bn(x)\one)^{-1}(\one-\<\psi_\bn(x),\cdot\>\psi_\bn(x))
$$
is localization preserving, uniformly in $x$ $($but not necessarily in $\bn$$)$.
\end{lem}
\begin{proof}
We have
\bee
\label{eq_Q_expansion}
R_{\bn}(x)=\sum_{\bn'\in \Z^d\setminus\{\bn\}}(E_{\bn'}(x)-E_{\bn}(x))^{-1}\<\psi_{\bn'}(x),\cdot\>\psi_{\bn'}(x).
\ene
In other words, in the basis $\{\psi_\bn(x)\colon \bn\in \Z^d\}$, $R_\bn(x)$ is a diagonal operator. From 
proposition \ref{prop_main_kps3}, the function $E$ satisfies the monotonicity condition \eqref{eq_monotonicity}:
\bee
\label{eq_E_monotonicity}
E(y)-E(x)\ge (y-x),\quad \text{for}\quad 0\le x\le y<1.
\ene
From \eqref{eq_dio_def} and \eqref{eq_eigenvalues} we now have
\bee
|E_{\bn'}(x)-E_{\bn}(x)|^{-1}=|E(x+\bn'\cdot\omega)-E(x+\bn\cdot\omega)|^{-1}\le|(\bn-\bn')\cdot\omega|\le  C_\dio^{-1}|\bn-\bn'|^{\tau_\dio}.
\ene
As a consequence of Proposition \ref{prop_localization_preserving_examples}, the diagonal operator with such entries is localization preserving. On the other hand, from \eqref{eq_uniformly_localized} it follows that both the operator of change of basis from $\{e_\bn\}$ to $\{\psi_\bn(x)\}$ and its inverse are also localization preserving. As a consequence, $R_\bn(x)$ is a composition of localization preserving operators and therefore is itself localization preserving.
\end{proof}
\begin{rem}
\label{rem_uniformity}
The fact that the conclusion of Lemma \ref{lemma_localization_preserving} is not uniform an $\bn$ is a consequence of the fact that the definition of localization preserving is stated in terms of localization relative to the origin. A uniform version can be obtained by an appropriate restatement that considers $\bn$ as the localization center. However, the same result can also be obtained by using the covariance relation
$$
R_{\bn}(x)=T^{\bn}R_{\bze}(x+\bn\cdot\omega)T^{-\bn}. 
$$
\end{rem}
\begin{cor}
\label{cor_perturbation}
The eigenvalue and eigenvector branches of $H(x)$ constructed in Proposition $\ref{prop_main_kps3}$ are infinitely smooth in $\varepsilon$, where the smoothness in the case of eigenvectors is considered in terms of the operator norm. The derivatives can be calculated as partial sums of the corresponding Rayleigh -- Schr\"odinger series.
\end{cor}
\begin{proof}
The assumptions of Corollary \ref{cor_uniform_smoothness} are satisfied due to Lemma \ref{lemma_localization_preserving} and Proposition \ref{prop_main_kps3}. In particular, the non-orthogonality condition follows from \eqref{eq_uniformly_localized}.
\end{proof}
In a similar fashion, we also obtain smoothness under $C^1$-perturbations of $f$. Let $g\colon \R\to \R$ be a $1$-periodic function, such that $\|g\|_{C^1[0,1]}\le 1$, and let $0<t_0<1$. 
Then, perhaps after a modification of $\ep_0$ depending on $t_0$, the operators $H_{f+tg}$ with $f$ replaced by $f+tg$ satisfy the assumptions of Proposition $\ref{prop_main_kps3}$, with all estimates uniform in $t\in [-t_0,t_0]$ $($in view of Remark $\ref{rem_monotonicity_constant}$$)$.
\begin{cor}
\label{cor_smoothness_g}
 In the above notation, the eigenvalues \eqref{eq_eigenvalues} and eigenvectors \eqref{eq_eigenvectors}, considered as functions of $t$, are infinitely smooth.
\end{cor}
\begin{rem}
We note that the actual calculation of the derivatives at $\ep\neq 0$ is complicated by the fact that one needs to know the original eigenvalues and eigenvectors at the point where the derivative is considered. Unlike the earlier results \cite{KPS} for Maryland-type potentials, Proposition \ref{prop_main_kps3} only establishes existence of these eigenvectors with uniform localization bounds. However, one can calculate any derivative at $\varepsilon=0$, in which case the result will be given, not surprisingly, by exact same formulas as in \cite{KPS}. The main difference with \cite{KPS} is that in the latter one has analyticity of eigenvalues and eigenvectors, which is equivalent to an $O(\ep^{N+1})$ estimate for the remainder. In the present case, the estimate on the remainder is only $C(N)\ep^{N+1}$, although the regularity assumptions on $f$ are much less restrictive. We do not know whether the branches constructed in Proposition \ref{prop_main_kps3} are analytic in $\ep$.
\end{rem}
\subsection{Gap labeling and zeros of eigenfunctions}
We say that $f$ is {\it sawtooth-type} defined already if it is continuous on $[0,1)$ and $-\infty<f(0)<f(1-0)<+\infty$. In this case, one can define an operator $H_{f_t}(x)$, with $f$ replaced by $f_t$:
\bee
\label{eq_t_def}
f_t(x)=\begin{cases}
f(x),&x\in (0,1)\\
t,&x=0,
\end{cases};\quad t\in (\R\cup\infty)\setminus(f(0),f(1-0)).
\ene
extended periodically into $\mathbb R$. The choice of the range of values of $t$ is made in order for $f_t$ to preserve the monotonicity property \eqref{eq_monotonicity} (perhaps after changing the continuity interval from $[0,1)$ to $(0,1]$). We note that the case $t=\infty$ is allowed, and refer the reader to \cite[Section 3]{KPS2} for a detailed discussion of this case. 

Clearly, the presence of the parameter $t$ only matters for the operator $H_{f_t}(x)$ with $x\in \Z+\Z^d\cdot\omega$. Since the change $x\mapsto x+\bn\cdot\omega$ corresponds to a translation by a vector $\bn\in \Z^d$, we will focus our attention on the zero phase operator $H_{f_t}:=H_{f_t}(0)$. In this case, as $t$ changes, the operator $H_{f_t}$ undergoes a rank one perturbation:
\bee
\label{eq_rank_one}
H_{f_t}-H_0=(t-f(0))\<e_\bze,\cdot\>e_\bze,
\ene
which typically leads to eigenvalues appearing in spectral gaps. We refer the reader to the general discussion in \cite[Sections 4.1 and 4.2]{KPS2}. In a somewhat more general setting, as $t$ runs over the range defined in \eqref{eq_t_def} in a ``counter-clockwise'' direction, an eigenvalue will at some point appear at the lower edge of each gap and later disappear at the upper edge. However, the general theory does not predict the exact moment of appearance and disappearance.

In \cite[Theorem 1.4]{KPS3} it was shown that, under the assumptions of Proposition \ref{prop_main_kps3}, the spectrum of the operator $H$ has infinitely many gaps. The following result can be considered as a refinement that shows that there is a natural correspondence between labels of the gaps and labels of the eigenfunctions, and whether the gap is open is controlled by the value of the corresponding eigenfunction at the origin. We also show that, as $t$ runs over the range \eqref{eq_t_def}, the eigenvalues in all gaps appear at the earliest possible moment and disappear at the latest possible moment.
\begin{thm}
\label{th_gap_labeling}
Under the assumptions of Proposition $\ref{prop_main_kps3}$, assume that $f$ is sawtooth-type. Then, the following are true.
\begin{enumerate}
	\item The function $x\mapsto E(x)$ is continuous for $x\in [0,1)\setminus\{\omega\cdot\Z^d\}$. As a consequence, the only allowed values of the IDS at spectral gaps are $\{\bn\cdot\omega\}$, with $\bn\in \Z^d$.
	\item The gap associated to $\bn\in \Z^d$ in the above sense is open if and only if $\psi_{\bn}(0;\bze)\neq 0$. In other words, if and only if the zero phase eigenfunction with label $\bn\in \Z^d$ does not vanish at the origin.
	\item For $t\in (\R\cup\infty)\setminus[f(0),f(1-0)]$, all eigenvalues of $H_{f_t}(0)$ are isolated, and there is exactly one eigenvalue in each gap.
	\item The above isolated eigenvalues are infinitely smooth in $t$ on $(-\infty,f(0)]$ and $[f(1-0),+\infty)$.
\end{enumerate}
\end{thm}
Note that in the fourth property it is already well known that the eigenvalues will be real analytic in $t$ once they separate from the gap edges. The new part is smoothness as they approach the gap edges.
\begin{proof}
From the gap filling theorem (see \cite[Theorem 4.2]{KPS2}), it follows that the set of isolated eigenvalues of $H_{f_t}(0)$, as runs over the range \eqref{eq_t_def}, will fill the whole real line. On the other hand, each of these isolated eigenvalues must be between two consecutive values of $E(\{\bn\cdot\omega\})$, which are the eigenvalues of the zero phase operator $H_{f_0}(0)=H(0)$. As a consequence, the value of the IDS at each gap must be equal to that at the energy $E(\{\bn\cdot\omega\})$, which implies the first claim.

If $\<\psi_{\bn}(0),e_\bze\>=0$, then the same $\psi_\bn(0)$ will be an eigenfunction of $H_{f_t}$ for all values of $t$. As a consequence, the corresponding eigenvalue will not be moving with $t$, which would contradict the gap filling results \cite[Theorem 4.2]{KPS2} if that gap was open. Therefore, it must be closed.

Now, let us note that the family of rank one perturbations \eqref{eq_rank_one} satisfies the assumptions of Theorem \ref{th_smoothness} and Corollary \ref{cor_uniform_smoothness}. As a consequence, if $\<\psi_{\bn}(0),e_\bze\>\neq 0$, then the first order calculation \eqref{eq_E_first_order} implies that the corresponding eigenvalue has positive derivative, and therefore must start moving as soon as $t$ departs from its initial value $f(0)$ in the negative direction. The same applies to $t$ moving upwards from $f(1-0)$. As a consequence, in both cases the eigenvalue appears and disappears immediately at the same time as $t$ leaves its boundary value.

We also note that, from general perturbation theory \cite{Kato}, the corresponding isolated eigenvalues are analytic in $t\in (\R\cup\infty)\setminus[f(0),f(1-0)]$ (as well as at $t=\infty$ if defined appropriately). In addition to that, we now know that they are infinitely smooth as $t$ approaches $f(0)$ or $f(1-0)$ from the correct side.
\end{proof}
As a consequence, we can calculate exactly the measure of the spectrum of a sawtooth-type operator \eqref{eq_h_def} in the same regime.
\begin{thm}
\label{th_measure_of_spectrum}
Under the assumptions of Theorem $\ref{th_gap_labeling}$, the measure of the spectrum of the operator \eqref{eq_h_def} is equal to $f(1-0)-f(0)$.
\end{thm}
\begin{proof}
Define $H_t$ as in \eqref{eq_t_def}, \eqref{eq_rank_one}, and consider each eigenvalue \eqref{eq_eigenvalues} of $H_{f_t}$ as a function of $t$, denoted by $E_{\bn}(0;t)=E(\bn\cdot\omega;t)$.
Here, $0$ reflects the zero phase. Clearly, each $E_{\bn}(0;t)$ is a non-decreasing function of $t$, strictly increasing if and only the gap with label $\bn$ is open, and constant if and only if the corresponding gap is closed. By comparing the traces of the operators, we have
\bee
\label{eq_trace_formula}
\sum_{\bn\in\Z^d}E_\bn'(0;t)=1,\quad \text{for}\quad t\in \R\setminus(f(0),f(1-0))=:\R\setminus J,
\ene
where prime denotes the derivative in $t$. We also note that, from Theorem \ref{th_gap_labeling},
$$
\sigma(H(0))=\R\setminus\cup_{\bn\in \Z^d}\{E_\bn(0;t)\colon t\in (\R\cup\infty)\setminus[f(0),f(1-0)],E_\bn(0;t)\not\equiv \mathrm{const} \}.
$$
Since the ranges of different $E_\bn(0;t)$ do not overlap, it is tempting to claim that the total length of these ranges must be equal to the length of the range of $t$ and conclude the proof by comparing the Lebesgue measures of the complements of each. Of course, one has to be somewhat careful since the above argument involves subtracting infinite quantities.

In order to be more precise, note that for $\bn\in \Z^d\setminus\{\bze\}$, each $E_\bn(0;\cdot)$ is bounded for the range of $t$ under consideration. The only unbounded eigenvalue is $E_\bze$, which corresponds to the eigenvalue contained in the gap of infinite length. From Theorem \ref{th_lattice_diagram_expansion}, we have
$$
E_\bze(0;t)-t\le \frac{C}{1+|t|}.
$$
Indeed, note that in the first order expansion \eqref{eq_E_first_order}, the first term vanished due to absence of diagonal elements in the Laplacian, and each term in the remainder has a copy of the operator $Q$, with the corresponding vertex factor containing $(t-O(1))^{-1}$. As a consequence, for a large $M$ one can integrate \eqref{eq_trace_formula} in $t$ restricted to $[-M,M]$ and obtain
\begin{multline}
\label{eq_trace_formula_integrated}
\sum_{\bn\in \Z^d\setminus \bze}\int_{[-M,M]\setminus J}E'_\bn(0;t)\,dt=\int_{[-M,M]\setminus J}\l((1-E_\bze'(0;t)\r)\,dt=\\=E_\bze(f(1-0))-E_\bze(f(0))-(f(1-0)-f(0))+O(M^{-1}).
\end{multline}
We note that $E_\bze(f(1-0))$ and $E_\bze(f(0))$ are, respectively, the upper and lower edges of $\sigma(H(0))$, and the left hand side of \eqref{eq_trace_formula_integrated} converges to the sum of the length of all finite spectral gaps as $M\to +\infty$, which completes the proof.
\end{proof}

\subsection{Opening and closing some of the gaps}
\label{subsection_principal_gaps}
Clearly, the gap with label $\bze$ is the gap containing $\infty$, and it is always open for sawtooth-type potentials. A gap label $\bn\in\Z^d$ will be called {\it principal} if it has exactly one non-zero component. In other words, a principal label associated to the standard basis vector $\bb_j$ of $\Z^d$ is 
$$
\bn=(0,\ldots,0,n,0,\ldots,0)=n\bb_j,\quad n\in \Z\setminus \{0\},\quad j=1,\ldots,d.
$$
We can now combine Corollary \ref{cor_perturbation} with Theorem \ref{th_gap_labeling} and arrive to the following conclusion.
\begin{cor}
\label{cor_principal_gaps}
Under the assumptions of Theorem $\ref{th_gap_labeling}$, fix a principal label of the form $\bn=n\bb_{j}$ for $n\in \Z\setminus \bze$. Then, there exists $\varepsilon_0=\varepsilon_0(d,C_{\dio},\tau_\dio)$ such that the gap associated with the above label is open for $0<\ep<\ep_0$.
\end{cor}
\begin{proof}
From Theorem \ref{th_gap_labeling}, it is sufficient to show that one can prevent $\psi_\bn(0)$ from vanishing at the origin. Since there is exactly one lattice path $\cP$ from $\bn$ to $\bze$, one can apply Theorem \ref{th_lattice_diagram_expansion} in view of Corollary \ref{cor_perturbation} for $N=n+1$, which will imply
$$
|\psi_\bn(0)-\ep^n\cont(\cP)|\le C(n)\ep^{n+1}.
$$
From the construction, $\cont(\cP)\neq 0$, which implies $\psi_\bn(0)\neq 0$ for small $\ep$. Note, however, that the smallness of $\ep$ depends on $n$, and therefore this argument can only be used to open finitely many principal gaps.
\end{proof}
We note that the conclusion of Corollary \ref{cor_principal_gaps} is the consequence of the fact that the leading term in the asymptotic expansions associated to gaps with principal labels has only one lattice path involved. We expect that this extends to some more general, not necessarily monotone models. 

Let us now call labels of the form $\pm\bb_i\pm \bb_j$, where $i\neq j$, {\it semi-principal labels}. In other words, $\bn\in \Z^d$ is a semi-principal label if and only if $\bn$ is not principal and $|\bn|_1=2$. In this case, there are precisely two lattice paths that contribute to the corresponding eigenvector component. The following corollary, which is even more specific to monotone $f$, demonstrates that semi-principal labels share the same non-closing property.
\begin{cor}
\label{cor_small_labels}
The conclusion of Corollary $\ref{cor_principal_gaps}$ also holds for gaps with semi-principal labels.
\end{cor}
\begin{proof}
Without loss of generality, one can assume $d=2$, $\bn=(1,1)$, and $0<\omega_1<\omega_2<1$. In this case, we have the expansion
\begin{multline}
\label{eq_second_gap}
\psi_{(1,1)}(0;\bze)=\ep^2\frac{1}{f(\omega_1+\omega_2)-f(0)}\l(\frac{1}{f(\omega_1+\omega_2)-f(\omega_1)}+\frac{1}{f(\omega_1+\omega_2)-f(\omega_2)}\r)+O(\ep^3)\\
=\ep^2 \frac{2 f(\omega_1+\omega_2)-f(\omega_1)-f(\omega_2)}{(f(\omega_1+\omega_2)-f(0))(f(\omega_1+\omega_2)-f(\omega_1))(f(\omega_1+\omega_2)-f(\omega_2))}+O(\ep^3).
\end{multline}
The reader can easily check that, due to monotonicity of $f$, the numerator must be strictly negative if $\omega_1+\omega_2>1$ and strictly positive if $0<\omega_1+\omega_2<1$.
\end{proof}
All remaining labels (that is, neither principal nor semi-principal) will be called {\it non-principal}. In the following two ``minimal'' examples, we illustrate the fact that the gaps with non-principal labels may close under a careful choice of Diophantine frequencies and a small perturbation of $f(x)=\{x\}$, for arbitrarily small $\ep>0$.
\begin{example}
\label{example_label12}
Let us consider the expansion for $\psi_{(1,2)}(0;\bze)$. There are three shortest paths connecting the lattice point $(1,2)$ to the origin. For the associated expansion, we obtain
\begin{multline}
\label{eq_gap12}
\psi_{(1,2)}(0;\bze)=\frac{\ep^3}{(f(\omega_1+2\omega_2)-f(0))}\l\{\frac{1}{f(\omega_1+2\omega_2)-f(\omega_1+\omega_2)}\frac{1}{f(\omega_1+2\omega_2)-f(\omega_1)}\r.\\+\frac{1}{f(\omega_1+2\omega_2)-f(\omega_1+\omega_2)}\frac{1}{f(\omega_1+2\omega_2)-f(\omega_2)}\\ \l.+\frac{1}{f(\omega_1+2\omega_2)-f(2\omega_2)}\frac{1}{f(\omega_1+2\omega_2)-f(\omega_2)}\r\}+O(\ep^4).
\end{multline}
In view if Corollary \ref{cor_smoothness_g} and Remark \ref{rem_monotonicity_constant}, the remainder denoted by $O(\ep^4)$ is uniformly bounded, assuming that the Diophantine class of $\omega$ and the monotonicity constant of $f$ are fixed. Moreover, within the above constraints, the remainder is continuous (and smooth) in the parameter $t$ considered in \ref{cor_smoothness_g}.

We will now make some special choice of frequencies. Let $\omega_1+\omega_2=1+\delta$, where $\delta$ will be picked later. For brevity, let
$$
F=F(\delta,f):=f(\omega_1+2\omega_2)-f(0)=f(\omega_2+\delta)-f(0).
$$
Then 
\begin{multline*}
\ep^{-3}F\psi_{(1,2)}(0;\bze)=\frac{1}{f(\omega_2+\delta)-f(\omega_2)}\l\{\frac{1}{f(\omega_2+\delta)-f(\delta)}+\frac{1}{f(\omega_2+\delta)-f(2\omega_2)}\r\}\\
+\frac{1}{(f(\omega_2+\delta)-f(\delta))(f(\omega_2+\delta)-f(\omega_1))}+\ep^{-3}FO(\ep^4)
\\=\frac{1}{f(\omega_2+\delta)-f(\omega_2)}\cdot \frac{2f(\omega_2+\delta)-f(2\omega_2)-f(\delta)}{(f(\omega_2+\delta)-f(\delta))(f(\omega_2+\delta)-f(2\omega_2))}\\
+\frac{1}{(f(\omega_2+\delta)-f(\delta))(f(\omega_2+\delta)-f(\omega_1))}+\ep^{-3}FO(\ep^4).
\end{multline*}
Assume that $0<\omega_2<2\omega_2+\delta<1$, and choose
\bee
\label{eq_choose_f}
f(x):=\{x\}+\gamma g(x),
\ene where $g(2\omega_2)=1$ and $g$ is a bump function supported on $\l[2\omega_2-\frac{1}{100},2\omega_2+\frac{1}{100}\r]$. We note that, for $\delta$ small enough, this guarantees that the values of $f$ at the remaining sampling points will not change. Then the latter expression simplifies to
\begin{multline}
\label{eq_g_estimate}
\ep^{-3}F\psi_{(1,2)}(0;\bze)=\\=-\delta^{-1}\frac{\delta-\gamma}{\omega_2(\omega_2-\delta+\gamma)}-\frac{1}{\omega_2(\omega_1-\omega_2-\delta)}+\ep^{-3}FO(\ep^4)=:G(\delta,\gamma)+FO(\ep),
\end{multline}
where we note that $F$, in fact, does not depend on $\gamma$.

In the middle expression of \eqref{eq_g_estimate}, for a fixed small $\delta$, all denominators are away from zero, and all three terms are continuous in $\gamma\in [-\delta^{1/2},\delta^{1/2}]$. The range of values of the first term contains a large interval $[-C_1\delta^{-1/2},C_1\delta^{-1/2}]$. Assuming that $\omega_1\neq \omega_2$, the second term is can be bounded in absolute value by a constant independent of $\delta$. The correction term can be estimated by $\ep C_2(\omega,\delta)$. Since $\ep$ can be chosen last, this term can be made arbitrarily small.

Since all three terms are continuous in $\gamma$ and the estimates are uniform in $\gamma \lesssim O(\delta^{1/2})$, there exists a choice of $\gamma$ that would assign the right hand side of \eqref{eq_g_estimate} any value in the range say, $\l[-\frac12 C_1\delta^{-1/2},\frac12 C_1\delta^{-1/2}\r]$, which includes the origin.

Since the order of choice of the parameters plays a key role in this argument, we list the precise order here for the convenience of the reader.
\begin{enumerate}
    \item Choose a Diophantine $\omega_2$ which is, say, close to $1/4$.
    \item Choose $\delta\in \l(0,\frac{1}{1000}\r)$ sufficiently small such that:\begin{itemize}
        \item The frequency vector $(\omega_1,\omega_2)=(1+\delta-\omega_2,\omega_2)$ is Diophantine.
        \item The range of values of $F(\delta,\gamma)$, considered as a function of $\gamma\in [-\delta^{1/2},\delta^{1/2}]$, contains an interval $[-1,1]$.
        \item The function 
        \eqref{eq_g_estimate} satisfies the lower Lipschitz bound \eqref{eq_rem43} with $c=1/2$, uniformly in $\gamma$ in the above range.
    \end{itemize}
    We note that the above choice of $\delta$ already implies the fact that the addition of $g$ does not change the values of $f$ at $\omega_1$, $\omega_2$, $\omega_1+\omega_2$, and $\omega_1+2\omega_2$.
    \item Choose $\ep_0>0$ small enough such that the remainder $F(\delta,\gamma)O(\ep)$ is bounded by $1/2$ uniformly in $\gamma$ in the above range (see Corollary \ref{cor_perturbation}).
\end{enumerate}
\end{example}
\begin{example}
\label{example_closing111}
We now consider the expansion of $\psi_{(1,1,1)}(0;\bze)$. In this case, there are six lattice paths. Similarly to the previous example, let
$$
F=F(\delta,f):=f(\omega_1+\omega_2+\omega_3)-f(0);\quad \Omega:=\omega_1+\omega_2+\omega_3.
$$
Then,
\begin{multline*}
\ep^{-3}F\psi_{(1,1,1)}(0;\bze)=\frac{1}{f(\Omega)-f(\omega_1+\omega_2)}\l\{\frac{1}{f(\Omega)-f(\omega_1)}+\frac{1}{f(\Omega)-f(\omega_2)}\r\}\\
+\frac{1}{f(\Omega)-f(\omega_1+\omega_3)}\l\{\frac{1}{f(\Omega)-f(\omega_1)}+\frac{1}{f(\Omega)-f(\omega_3)}\r\}\\+
\frac{1}{f(\Omega)-f(\omega_2+\omega_3)}\l\{\frac{1}{f(\Omega)-f(\omega_2)}+\frac{1}{f(\Omega)-f(\omega_3)}\r\}+\ep^{-3}FO(\ep^4).
\end{multline*}
Let $\omega_1+\omega_2=1+\delta$. One can now rewrite
\begin{multline*}
\ep^{-3}F\psi_{(1,1,1)}(0;\bze)=\frac{1}{f(\omega_3+\delta)-f(\delta)}\l\{\frac{1}{f(\omega_3+\delta)-f(\omega_1)}+\frac{1}{f(\omega_3+\delta)-f(\omega_2)}\r\}\\
+\frac{1}{f(\omega_3+\delta)-f(\omega_1+\omega_3)}\l\{\frac{1}{f(\omega_3+\delta)-f(\omega_1)}+\frac{1}{f(\omega_3+\delta)-f(\omega_3)}\r\}\\+
\frac{1}{f(\omega_3+\delta)-f(\omega_2+\omega_3)}\l\{\frac{1}{f(\omega_3+\delta)-f(\omega_2)}+\frac{1}{f(\omega_3+\delta)-f(\omega_3)}\r\}+\ep^{-3}FO(\ep^4)\\
=\frac{1}{f(\omega_3+\delta)-f(\omega_3)}\cdot\frac{2f(\omega_3+\delta)-f(\omega_2+\omega_3)-f(\omega_1+\omega_3)}{(f(\omega_3+\delta)-f(\omega_1+\omega_3))(f(\omega_3+\delta)-f(\omega_2+\omega_3))}\\+\text{remaining terms}+\ep^{-3}FO(\ep^4).
\end{multline*}
Now, the reader can easily check that the expression behaves similarly to the previous example: assuming $f(x)=\{x\}$ and 
$$
0<\omega_1+\omega_3<1,\quad 0<\omega_2+\omega_3<1,
$$
one has $\delta$ in the numerator of the first term. By using a bump function perturbation supported near $\omega_2+\omega_3$, one can assign this term a large range of values, with uniform bounds on the remaining terms and the remainder.
\end{example}
\section{Generic Cantor spectra}
In this section, we will continue to work in the setting of Proposition \ref{prop_main_kps3}. We already know that, in the case of sawtooth-type potentials for $d\ge 2$, one has infinitely many gaps. Moreover, one can force (any) finite number of principal gaps to be open by taking a small $\ep$. So far, the arguments fall short from establishing Cantor spectra. Our goal will be to apply quantitative perturbation theory to the function $f$ itself. All quasiperiodic objects operators in this section will be considered at zero phase, and we will drop this from the notation:
$$
H_f=H_f(0),\quad R_\bn=R_\bn(0),\quad \psi_\bn(\bk)=\psi_\bn(0;\bk).
$$
The following is the main result of this section and the final main result of the paper.
\begin{thm}
\label{th_main_cantor}
Under the assumptions of Proposition $\ref{prop_main_kps3}$, for every $\varkappa>0$, there exists $g=g(f,\omega,\varkappa)\in C^1[0,1]$, $\|g\|_{C^1}\le \varkappa$, such that the eigenvectors of $H_{f+g}$ have no zeros. As a consequence, all spectral gaps of the corresponding operator in the sense of Theorem $\ref{th_gap_labeling}$ will be open.
\end{thm}

\begin{rem}
\label{remark_ck}
The reader can easily modify the proof and replace $C^1$ by $C^k$, or even by a norm that involves infinitely many derivatives.
\end{rem}

\begin{rem}
\label{remark_generic}
Theorem \ref{th_main_cantor} can be restated as a generic result. For example, let us consider the topology of uniform convergence on the sampling functions $f$. From the usual operator norm perturbation arguments, for every $\bn\in\Z^d\setminus\{\bze\}$, the set of sawtooth-type sampling functions for which the gap with label $\bn$ is open, is open. Theorem \ref{th_main_cantor} shows that it is also dense. We note that, while the addition of $g$ may decrease the monotonicity constant, the argument can be applied to $(1+\ep) f$ instead of $f$, and then require $\|g\|_{C^1}\ll\ep$.

As a consequence, the set of sawtooth-type sampling functions $f$ which have all gaps open is a dense $G_{\delta}$ subset in the topology of uniform convergence.

We note that the claim can be made even stronger if one allows variations in the monotonicity constant. For example, let us consider sampling functions of the form
\bee
\label{eq_monotonicity_modified}
f(y)-f(x)\ge \alpha(y-x),\quad\text{for}\quad 0\le x<y<1\quad \text{and some}\quad\alpha>1/2.
\ene
Within the class of bounded sampling functions satisfying \eqref{eq_monotonicity_modified}, the set of those with all spectral gaps open will be a dense $G_\delta$-set in $C^k$-topology for every $k\in \N$. We note that $f$ itself is not assumed to be smooth, and the introduction of $C^k$-topology results in the above set of sawtooth-type potentials being ``very disconnected''.
\end{rem}

Before proving Theorem \ref{th_main_cantor}, we will need some preparatory work. Our first goal would be to construct an arbitrary small perturbation of $f$ that will change the value of an individual eigenfunction at an individual point from zero to non-zero. Let $g\colon \R\to [0,1]$ be a bump function: 
$$
g(0)=1;\quad g(x)=0\quad\text{for}\quad |x|\ge 1/4;\quad |g'|\le 5.
$$
Let also 
$$
g_{\eta}(x):=g(x/\eta),\quad 0<\eta<1.
$$
Finally, we assume that $g_{\eta}$ is extended from $[-1/2,1/2]$ into $\R$ by $1$-periodicity.

Suppose that $f$ and the operator $H_f$ satisfy the assumptions of Proposition \ref{prop_main_kps3}. Then $f+tg$ is Lipschitz monotone with $c=1/2$ (see Remark \ref{rem_monotonicity_constant}) for, say, $|t|\le  \eta/4$. As a consequence, all results from the previous section apply uniformly in $t$ within the said range. Denote
$$
A_{\bk,\eta}:=\sum_{\bn\in \Z^d} g_\eta((\bn-\bk)\cdot\omega) \<e_\bn,\cdot\>e_\bn.
$$
Clearly, $A_{\bk,\eta}$ is diagonal. We also have
\bee
\label{eq_A_properties}
A_{\bk,\eta}e_\bk=e_\bk;\quad A_{\bk,\eta}e_\bn=0\quad\text{for}\quad 0<|\bn-\bk|\le C_\dio'|\eta|^{-1/\tau_\dio}.
\ene
\begin{lem}
\label{lemma_make_nonzero}
Under the assumptions of Proposition $\ref{prop_main_kps3}$, define $g$ as above and $R_{\bn}(x)$ as in Lemma $\ref{lemma_localization_preserving}$. Suppose that, for some $\bk,\bl\in \Z^d$, we have
\bee
\label{eq_q_assumption}
\<R_{\bn}e_{\bk},e_{\bl}\>\neq 0,\quad \psi_\bn(\bk)\neq 0.
\ene
Then, for small enough $\eta>0$, we have
\bee
\label{eq_q_conclusion}
\<R_\bn A_{\bk,\eta}\psi_\bn,e_{\bl}\>\neq 0.
\ene
As a consequence, the value $\psi_{\bn}(\bl)$ of the eigenfunction $\psi_\bn$ of $H_{f+tg}$ , considered as a scalar function of $t$, has a non-zero derivative at the origin.
\end{lem}
\begin{proof}
Since $\psi$ is localized, the family $A\psi$ is uniformly localized in the sense of Proposition \ref{prop_localization_preserving}, uniformly in $\eta$. Moreover, from \eqref{eq_A_properties} we have
$$
\l\<A_{\bk,\eta}\psi_{\bn}-\psi_\bn(\bk)e_\bk,e_\bm\r\>=0\quad \text{for}\quad|\bm-\bk|\le C_\dio'|\eta|^{-1/\tau_\dio}.
$$
For a fixed $\bk$ and small $\eta$ (depending on $\bk$), this would mean that $A_{\bk,\eta}\psi_{\bn}(\bze)-\psi_\bn(0;\bk)e_\bk$ is supported outside, say, a $C_\dio'|\eta|^{-1/\tau_\dio}/2$-neighborhood of the origin. From Proposition \ref{prop_localization_preserving}, we have
$$
R_\bn\l(A_{\bk,\eta}\psi_{\bn}-\psi_\bn(\bk)e_\bk\r)\to 0\quad \text{as}\quad \eta\to 0.
$$
In view of \eqref{eq_q_assumption}, this implies \eqref{eq_q_conclusion} for small $\eta$, and the claim regarding the eigenfunction now follows from \eqref{eq_psi_first_order} and Corollary \ref{cor_uniform_smoothness}.
\end{proof}
\begin{rem}
\label{remark_heuristic}
The heuristic calculation discussed in the Introduction is based on considering
$$
g(x)=\begin{cases}
    1,&x=0;\\
    0,&x\neq 0,
\end{cases}
$$
which corresponds to perturbing $f$ at a single point and the associated operator being $A_{\bk}=\<\cdot,e_\bk\>e_\bk$, up to normalization. In this case, \eqref{eq_q_conclusion} immediately follows from \eqref{eq_q_assumption}. However, \eqref{eq_psi_first_order} is clearly not applicable, since the perturbed eigenfunction is not known to exist.
\end{rem}
Lemma \ref{lemma_make_nonzero} can be understood as a way to ``transfer'' the mass from $\psi_\bn(0;\bk)$ into $\psi_\bn(0;\bl)$, using the non-zero matrix element $\<R_\bn(0)e_\bk,e_{\bl}\>$ as a ``bridge''. The next lemma establishes that one can reach every lattice point by a finite sequence of such bridges.
\begin{lem}
\label{lemma_cyclic}
Under the assumptions of Proposition $\ref{prop_main_kps3}$, fix $\bn\in \Z^d$ and consider the graph whose vertices are lattice points of $\Z^d$ and there is an edge between $\bk$ and $\bl$ if and only if $\<R_{\bn}e_{\bk},e_{\bl}\>\neq 0$. Then, for every $\bn\in \Z^d$ the corresponding graph is connected.
\end{lem}
\begin{proof}
Suppose that $A\subset\Z^d$ is a non-trivial connected component of the above graph. Then
$$
\<R_{\bn}e_{\bk},e_{\bl}\>=0,\quad \forall\, \bk\in A,\quad\forall\,\bl\in \Z^d\setminus A.
$$ 
If true, the same must therefore hold for every polynomial and every Borel function of $R_{\bn}$, and consequently for the original operator $H_f$ and all its powers. However, one can easily check that, for every $\bk\neq \bl$, there exists $N$ such that $\<H(0)^N e_{\bk},e_{\bl}\>\neq 0$. In fact, it is easy to see that this would be true for $N=|\bk-\bl|_{1}$.
\end{proof}
In order to reach a particular lattice point, one may have to use a path of large length on the above graph, which will result in repeated applications of Lemma \ref{lemma_make_nonzero}. Since each application changes the operator, one has to be careful in order to make sure that the chosen path is still available after such change. The key observation is the following lemma.
\begin{lem}
\label{lemma_iteration}
Fix $g$ as above and let $R_{t,\bn}$ be the operator $R_\bn$ considered with the function $f+tg$. Then, for every $\bk,\bl\in \Z^d$, we have
$$
\<R_{t,\bn}e_{\bk},e_{\bl}\>\to \<R_{0,\bn}e_{\bk},e_{\bl}\>=\<R_{\bn}e_{\bk},e_{\bl}\>,\quad t\to 0.
$$
\end{lem}
\begin{proof}
Follows from the expansion \eqref{eq_Q_expansion} and the dominated convergence theorem.
\end{proof}
{\noindent\it Proof of Theorem $\ref{th_main_cantor}$}.
For $j=1,2,\ldots$ we will construct an operator $H^{(j)}=H_{f_j}$, where $f_j$  will be a small perturbation of $f_{j-1}$, and $f_0=f$. The eigenvalues and eigenfunctions of the corresponding operator will be denoted by $E_\bn^{(j)}$, $\psi_\bn^{(j)}$, respectively. We start from $j=0$.
\begin{enumerate}
	\item Pick $\bn,\bm$ with $|\bn|,|\bm|\le j+1$ and $\psi_{\bn}^{(j)}(\bm)=0$.
	\item Apply Lemma \ref{lemma_cyclic} and find a path on the graph from that lemma connecting $\bn$ with $\bm$. Apply Lemma \ref{lemma_make_nonzero} in order to make all values of $\psi_\bn^{(j)}(\bk)$ at the points $\bk$ of the path to be non-zero. At each step, make sure that $t$ is chosen small enough so that the entries $\psi_\bn^{(j)}(\bm)$ with $|\bm|,|\bn|\le j+1$ that are non-zero remain non-zero. Moreover, at each step use Lemma \ref{lemma_iteration} in order to make sure that the matrix elements of the corresponding operator $R$ remain non-zero along the path.
	\item As a result of the previous step, we have transformed $\psi_\bn^{(j)}(\bm)$ from zero to non-zero. Repeat with all remaining $\bn,\bm$ with $|\bn|,|\bm|\le j+1$. Let $f_{j+1}$ be the result of all the above procedures.
	\item Note that the number of applications of Lemma \ref{lemma_make_nonzero} is finite at each step, and during each application we can impose, if needed, additional conditions on the smallness of $t$. These additional conditions will be summarized separately.
\end{enumerate}
The following are the additional conditions.
\begin{enumerate}
	\item For $|\bn|\le j+1$, we have 
$$
\|\psi_{\bn}^{(j+1)}-\psi_{\bn}^{(j)}\|+|E_\bn^{(j+1)}-E_\bn^{(j)}|+\|H^{(j+1)}-H^{(j)}\|\le 10^{-j}.
$$
\item Let $a_{\bn}^{(j)}(\bm)$ be the first obtained non-zero value of $\psi_\bn^{(j)}(\bm)$ with $|\bm|,|\bn|\le j+1$, either due to the eigenfunction being non-zero upon reaching the corresponding step, or as the result of the procedure. For $i\ge j$, we additionally impose
$$
|\psi_\bn^{(i+1)}(\bm)-\psi_\bn^{(i)}(\bm)|\le 10^{-i}|a_{\bn}^{(j)}(\bm)|.
$$
\item $\|f_{j+1}-f_j\|_{C^1[0,1]}\le \varkappa 10^{-j}$, where $\varkappa$ is the small parameter from the original statement of Theorem \ref{th_main_cantor}.
\end{enumerate}
At every step, these requirements present finitely many conditions on the smallness of $t$ at every application of Lemma \ref{lemma_make_nonzero}, and the reader can easily check that they can be imposed. As a result, each sequence of eigenfunctions $\psi^{(j)}$ converges in $\ell^2$, and the operators $H^{(j)}$ converge in the norm topology. By inspecting the eigenvalue equation, one can easily check
$$
H^{(\infty)}\psi_\bn^{(\infty)}=E_{\bn}^{(\infty)}\psi_\bn^{(\infty)},
$$
where the $\infty$ index denotes the limiting objects. Clearly, $f_{\infty}$ is a small $C^1$-perturbation of $f$, and the limiting eigenfunctions do not vanish anywhere.\,\,\qed

\begin{thebibliography}{99}
\bibitem{velocity}Abdul-Rahman~H., Fillman~J., Fischbacher~C., Liu~W., {\it Sharp Polynomial Velocity Decay Bounds for Multidimensional Periodic Schr\"odinger Operators}, preprint (2025), \url{https://arxiv.org/abs/2509.04381}.
\bibitem{ABD}Avila~A., Bochi~J., Damanik~D., {\it Opening gaps in the spectrum of strictly ergodic Schr\"odinger operators}. J. Eur. Math. Soc. (JEMS) 14 (2012), 61 -- 106.
\bibitem{ten_martini}Avila~A., Jitomirskaya~S., {\it The Ten Martini Problem}, Ann. of Math. (2) 170 (2009), no. 1, 303 -- 342.
\bibitem{ten_martini_dry}Avila~A., You~J., Zhou~Q., {\it Dry Ten Martini Problem in the non-critical case}, preprint (2024), \url{https://arxiv.org/abs/2306.16254}.
\bibitem{beckus}Band~R., Beckus~S., Loewy~R., {\it The Dry Ten Martini Problem for Sturmian Hamiltonians}, preprint (2024), \url{https://arxiv.org/abs/2402.16703}.
\bibitem{Bellissard} Bellissard~J., Lima~R., Scoppola~E., {\it Localization in $\nu$-dimensional incommensurate structures}, Comm. Math. Phys. 88 (1983), no. 4, 465 -- 477.
\bibitem{Bloch}Bloch~C., {\it Sur la th\'eorie des perturbations des \'etats li\'es}, Nuclear Physics 6 (1958), 329 -- 347.
\bibitem{BF}Borgnia~D., Faust~M., {\it The Dry Ten Martini Problem at Criticality}, preprint (2026), \url{https://arxiv.org/abs/2112.06869v2}.
\bibitem{Shi}Cao~H., Shi~Y., Zhang~Z., {\it Localization for Lipschitz monotone quasi-periodic Schr\"odinger operators on $\mathbb Z^d$ via Rellich functions analysis}, Commun. Math. Phys. 406 (2025), article no. 102.
\bibitem{ChuSi}Chulaevsky~V., Sinai~Y., {\it Anderson localization for the 1-D discrete Schr\"odinger operator with two-frequency potential}. Commun. Math. Phys. 125 (1989), 91 -- 112.

\bibitem{Craig_pure_point}Craig~W., {\it Pure point spectrum for discrete almost periodic Schr\"odinger operators}, Comm. Math. Phys. 88 (1983), no. 1, 113 -- 131.
\bibitem{CS_ids}Craig, W., Simon, B.: {\it Log H\"older continuity of the integrated density of states for stochastic Jacobi matrices}. Comm. Math. Phys. 90 (1983), 207 -- 218.
\bibitem{DF_label}Damanik~D., Fillman~J., {\it Gap Labelling for Discrete One-Dimensional Ergodic Schr\"odinger Operators}. In: Brown, M., et al. From Complex Analysis to Operator Theory: A Panorama. Operator Theory: Advances and Applications, vol 291. Birkh\"auser, Cham.
\bibitem{DL}Damanik~D., Li~L., {\it Opening Gaps in the Spectrum of Strictly Ergodic Jacobi and CMV Matrices}, J. Funct. Anal. 289 (2025), no. 12, article 111182.
\bibitem{DFG}Damanik~D., Fillman~J., Gorodetski~A., {\it Multidimensional almost-periodic Schr\"odinger operators with Cantor spectrum}, Ann. Henri Poincar\'e 20 (2019), 1393 -- 1402.
\bibitem{faust_k}Faust~M., Kachkovskiy~I., {\it Absence of flat bands for discrete periodic graph operators with generic potentials}, preprint (2025), \url{https://arxiv.org/abs/2509.01927}.
\bibitem{GJY}Ge~L., Jitomirskaya~S., You~J., {\it Kotani theory, Puig's argument, and stability of The Ten Martini Problem}, preprint (2024), \url{https://arxiv.org/abs/2308.09321}.
\bibitem{Ge2}Ge~L., Jitomirskaya~S., You~J., {\it The Robust Ten Martini Problem}, preprint (2026), https://arxiv.org/pdf/2609.15812.
\bibitem{GSW}Goldstein~M., Schlag~W., Voda~M., {\it On the spectrum of multi-frequency quasi-periodic Schr\"odinger operators at large coupling}, Invent. Math. 217 (2019), 603--701.
\bibitem{HLQZ}Hellfer~B., Liu~Q., Qu~Y., Zhou~Q., {\it Cantor spectrum for multidimensional quasi-periodic Schr\"odinger operators}, Forum of Mathematics, Sigma, Volume 14 (2026), e90.
\bibitem{JK}Jitomirskaya~S., Kachkovskiy~I., {\it All couplings localization for quasiperiodic operators with Lipschitz monotone potentials}, J. Eur. Math. Soc. 21 (2019), no. 3, 777 -- 795.
\bibitem{JK2}Jitomirskaya~S., Kachkovskiy~I., {\it Sharp arithmetic localization for quasiperiodic operators with monotone potentials}, preprint (2024), \url{https://arxiv.org/abs/2407.00703}.
\bibitem{KPS}Kachkovskiy~I., Parnovski~L., Shterenberg~R., {\it Convergence of perturbation series for unbounded monotone quasiperiodic operators}, Adv. Math. 409 (2022), part B, 108647.
\bibitem{KPS2}Kachkovskiy~I., Parnovski~L., Shterenberg~R., {\it On gaps in the spectra of quasiperiodic Schr\"odinger operators with discontinuous monotone potentials}, Int. Math. Res. Notes vol. 2025 (2025), Issue 7.
\bibitem{KPS3}Kachkovskiy~I., Parnovski~L., Shterenberg~R., {\it Perturbative diagonalization and spectral gaps of quasiperiodic operators on $\ell^2(\mathbb Z^d)$ with monotone potentials}, Comm. Math. Phys vol. 406 (2025), article number 125.
\bibitem{KKPS}Kachkovskiy~I., Krymski~S., Parnovski~L., Shterenberg~R., {\it Perturbative diagonalization for Maryland-type quasiperiodic operators with flat pieces}, J. Math. Phys. 62 (2021), 063509.
\bibitem{Cantor}Kachkovskiy~I., Parnovski~L., Shterenberg~R., {\it On gaps in the spectra of quasiperiodic Schr\"odinger operators with discontinuous monotone potentials}, IMRN, Vol. 2025, Issue 7, April 2025.
\bibitem{JK1}Jitomirskaya~S., Kachkovskiy~I., {\it All couplings localization for quasiperiodic operators with Lipschitz monotone potentials}, J. Eur. Math. Soc. 21 (2019), no. 3, 777 -- 795.
\bibitem{Kato}Kato~T., Perturbation Theory for Linear Operators, Springer, 1995.
\bibitem{Li}Li~X., Xu~D., Zhou~Q., {\it Monotonicity, global symplectification and the stability of Dry Ten Martini Problem}, preprint (2026), \url{https://arxiv.org/abs/2601.02222}.
\bibitem{Lippmann}Lippmann, B. A.; Schwinger, J. {\it Variational Principles for Scattering Processes. I}, Phys. Rev. Lett. 79 (1950), no. 3, 469 -- 480.
\bibitem{Sinai}Sinai~Y., {\it Anderson localization for one-dimensional difference Schr\"odinger operator with quasiperiodic potential}, J. Stat. Phys. 46 (1987), nos. 5/6, 861 -- 909.
\end{thebibliography}
\end{document}